\documentclass[a4paper,12pt]{article}
\usepackage{amsfonts,amsmath,amssymb,amsthm}
\usepackage{txfonts}
\usepackage[utf8]{inputenc}

\usepackage{amsmath}
\usepackage{amssymb}
\usepackage{mathtools}
\usepackage{bbm}

\RequirePackage[colorlinks,citecolor=blue,urlcolor=red, linkcolor=blue]{hyperref}
\usepackage[margin=0.8in]{geometry}

\usepackage[dvips]{graphics}
\usepackage{epsfig,rotating}
\usepackage{tabularx}
\usepackage{bm}
\usepackage{bbm}
\usepackage{array}
\newcolumntype{P}[1]{>{\raggedright\arraybackslash}p{#1}}

\usepackage{mdframed}
\usepackage{lipsum}
\usepackage{enumitem}

\usepackage[authoryear,square]{natbib}

\numberwithin{equation}{section}

\usepackage[nottoc,numbib]{tocbibind} 
\usepackage[toc,page]{appendix}

\newtheorem{thm}{Theorem}
\newtheorem{lemma}{Lemma}

\newtheorem{prop}{Proposition}
\newtheorem{co}{Corollary}

\DeclareMathOperator{\Var}{Var}

\newcommand{\beaa}{\begin{eqnarray*}}
\newcommand{\eeaa}{\end{eqnarray*}}
\newcommand{\bea}{\begin{eqnarray}}
\newcommand{\eea}{\end{eqnarray}}

\newcommand{\la}{\left\{}
\newcommand{\ra}{\right\}}
\newcommand{\lb}{\left(}
\newcommand{\rb}{\right)}
\newcommand{\mb}{\mathbb}
\newcommand{\ve}{\varepsilon}
\newcommand{\mc}{\mathcal}

\newcommand{\R}{\mathbb R}
\newcommand{\Pp}{\mathbb P}

\newcommand{\op}{\mathrm{op}}

\newcommand{\lmax}{\lambda_{\max}}

\newcommand{\mfQ}{\mathsf{Q}}

\renewcommand{\tabularxcolumn}[1]{%
    >{\centering\arraybackslash}m{#1}%
}

\usepackage{tikz}
\usetikzlibrary{arrows.meta}

\begin{document}
\title{Extreme principal minors of Wishart and deformed GOE matrices}
\author{ 
Zhaorui Dong \\ 
\small{School of Data Science, The Chinese University of Hong Kong, Shenzhen} \\ \small{\href{mailto: zhaoruidong@link.cuhk.edu.cn}{zhaoruidong@link.cuhk.edu.cn}} 
\and 
Tiefeng Jiang \\ \small{School of Data Science, The Chinese University of Hong Kong, Shenzhen} \\ \small{\href{mailto:jiang040@cuhk.edu.cn}{jiang040@cuhk.edu.cn}}
\and 
Tuan Pham \\ \small{Department of Statistics and Data Science, University of Texas, Austin} \\  \small{\href{mailto:tuan.pham@utexas.edu}{tuan.pham@utexas.edu} }
\and
Jianfeng Yao \\ 
\small{School of Data Science, The Chinese University of Hong Kong, Shenzhen} \\ \small{\href{mailto: jeffyao@cuhk.edu.cn}{jeffyao@cuhk.edu.cn} 
}}
\date{}

\maketitle

\begin{abstract}

We study the laws of large
numbers for the largest eigenvalues among all principal minors of
Wishart matrices and deformed GOE matrices. We propose a new method
based on identifying the deterministic sets to which the random sets
formed by suitably normalized principal minors converge in Hausdorff
distance, thereby reducing the original extreme-value problems to
finite-dimensional convex optimization problems.

We demonstrate the effectiveness of this method in regimes not covered
by the existing second-moment arguments in
\cite{cai2021asymptotic,hu2023extreme}. For deformed GOE matrices with
fixed minor size \(k\), we determine the limit for every diagonal
variance \(a>0\) and identify a phase transition at \(a=2\). Above the
transition, the limiting constant satisfies an explicit recursion with no close-form expression, and
the optimizers exhibit a nested hierarchical structure, thereby
resolving the case left open in \cite{cai2021asymptotic}. For Wishart
matrices with general sub-Gaussian entries and fixed \(k\), we
characterize the limit through an entropy-constrained deterministic
convex set. When the entries are standard Gaussian, we solve the
resulting optimization problem explicitly and obtain the exact value
of the limiting constant.

\end{abstract}






\tableofcontents

\section{Introduction}

A substantial part of the random matrix theory literature studies the
eigenvalues and eigenvectors of a single large random matrix; see, for
example,
\cite{anderson2010introduction,BaiSilverstein2010,Bai1999,
BrycDemboJiang2006,DiaconisEvans2001}
and the references therein. In many high-dimensional problems, however,
the relevant quantities are not determined by the spectrum of the full
matrix, but by spectral functionals of a combinatorial family of its
principal submatrices. In compressed sensing
\cite{candes2008restricted}, sparse principal component analysis
\cite{moghaddam2005spectral,d2004direct}, and high-dimensional linear
regression under sparsity constraints \cite{zhang2008sparsity}, either
the optimization problem itself or the assumptions guaranteeing its
statistical performance can be expressed in terms of the extreme
eigenvalues of many principal submatrices. We describe these connections
in more detail below.

\medskip
\noindent
\underline{\it Compressed sensing.}
Suppose that $\bm\Phi\in\mb R^{m\times n}$ is a sensing matrix. The
restricted isometry constant of order $s$ is the smallest
$\delta_s\geq0$ such that
\[
    (1-\delta_s)\|\bm x\|_2^2
    \leq
    \|\bm\Phi\bm x\|_2^2
    \leq
    (1+\delta_s)\|\bm x\|_2^2
\]
for every $s$-sparse vector $\bm x\in\mb R^n$.
The restricted isometry property is used to obtain uniform recovery
guarantees for sparse signals; see, for example,
\cite{candes2008restricted}.

For $S\subset[n]$, let
$\bm\Phi_S\in\mb R^{m\times |S|}$ denote the submatrix formed by the
columns indexed by $S$. The preceding inequalities hold for every
$s$-sparse vector if and only if
\[
    \min_{\substack{S\subset[n]\\1\leq |S|\leq s}}
    \lambda_{\min}
    \left(
        \bm\Phi_S^\top\bm\Phi_S
    \right)
    \geq
    1-\delta_s
    \qquad \text{and}
    \qquad
     \max_{\substack{S\subset[n]\\1\leq |S|\leq s}}
    \lambda_{\max}
    \left(
        \bm\Phi_S^\top\bm\Phi_S
    \right)
    \leq
    1+\delta_s.
\]
Since
$
    \bm\Phi_S^\top\bm\Phi_S
    =
    \left(
        \bm\Phi^\top\bm\Phi
    \right)_S,
$
the restricted isometry property is precisely a simultaneous spectral
condition on all principal submatrices of the Gram matrix
$\bm\Phi^\top\bm\Phi$ of order at most $s$. Random sensing matrices are
a standard and convenient way to obtain such uniform spectral control
with high probability; see
\cite{JiangCai12,cai2021asymptotic}.

\medskip
\noindent
\underline{\it Sparse principal component analysis.}
Given a covariance matrix
$\bm\Sigma\in\mb R^{p\times p}$, the sparse principal component
analysis problem seeks to solve
\[
    \max_{\substack{
        \|\bm x\|_2=1\\
        \|\bm x\|_0\leq s
    }}
    \bm x^\top\bm\Sigma\bm x.
\]
For a fixed support $S$, maximizing over unit vectors supported on
$S$ gives the largest eigenvalue of the corresponding principal
submatrix. Consequently,
\[
    \max_{\substack{
        \|\bm x\|_2=1\\
        \|\bm x\|_0\leq s
    }}
    \bm x^\top\bm\Sigma\bm x
    =
    \max_{\substack{
        S\subset[p]\\
        1\leq|S|\leq s
    }}
    \lambda_{\max}(\bm\Sigma_S);
\]
see, for example, \cite{moghaddam2005spectral,d2004direct} and the references therein. Thus, the sparse PCA
objective is exactly the largest eigenvalue among a collection of principal submatrices. In applications, the unknown population covariance matrix is usually replaced by the sample covariance matrix or by a regularized version of it. The corresponding statistical problem therefore requires understanding the extreme
eigenvalues of many strongly dependent random principal submatrices.

\medskip
\noindent
\underline{\it High-dimensional linear regression.}
Consider the linear model
\[
    \bm y
    =
    \bm X\bm\theta+\bm\varepsilon,
\]
where
$\bm X\in\mb R^{N\times p}$ is the design matrix and the regression
vector $\bm\theta\in\mb R^p$ is assumed to be sparse. Let
\[
    \widehat{\bm\Sigma}
    :=
    \frac1N\bm X^\top\bm X
\]
be the empirical Gram matrix. Its lower and upper sparse eigenvalues
of order $s$ are
\begin{align*}
    \phi_{\min}(s)
    &:=
    \min_{\substack{
        \|\bm v\|_2=1\\
        \|\bm v\|_0\leq s
    }}
    \bm v^\top
    \widehat{\bm\Sigma}
    \bm v
    =
    \min_{\substack{
        S\subset[p]\\
        1\leq|S|\leq s
    }}
    \lambda_{\min}
    \left(
        \widehat{\bm\Sigma}_S
    \right),\\
    \phi_{\max}(s)
    &:=
    \max_{\substack{
        \|\bm v\|_2=1\\
        \|\bm v\|_0\leq s
    }}
    \bm v^\top
    \widehat{\bm\Sigma}
    \bm v
    =
    \max_{\substack{
        S\subset[p]\\
        1\leq|S|\leq s
    }}
    \lambda_{\max}
    \left(
        \widehat{\bm\Sigma}_S
    \right),
\end{align*}
where
\[
    \widehat{\bm\Sigma}_S
    =
    \frac1N\bm X_S^\top\bm X_S.
\]
The sparse Riesz condition requires constants
$0<c_\star\leq c^\star<\infty$ such that
\[
    c_\star
    \leq
    \lambda_{\min}
    \left(
        \frac1N\bm X_S^\top\bm X_S
    \right)
    \leq
    \lambda_{\max}
    \left(
        \frac1N\bm X_S^\top\bm X_S
    \right)
    \leq
    c^\star
\]
uniformly over all subsets $S$ up to a prescribed size
\citep{zhang2008sparsity}. Conditions of this type prevent nearly
singular behavior on sparse directions and play a central role in
estimation-error, sparsity, and variable-selection guarantees for
the Lasso and related procedures. When the design matrix is random,
these guarantees depend on the simultaneous spectral behavior of a
large, strongly dependent family of random principal submatrices.

These examples motivate the study of extreme spectral functionals over
all principal minors, rather than the spectrum of a single random matrix.
Despite their broad relevance, however, asymptotic results for spectral
extrema over a family of principal minors remain relatively
limited, especially when compared with the extensive theory available for
the spectrum of a single random matrix. A principal difficulty is that one
must optimize simultaneously over a large collection of submatrices, which can be highly dependent due to the overlapping structure of the index sets.

A systematic asymptotic study of this problem was initiated in \cite{cai2021asymptotic} for white Wishart matrices and Gaussian Wigner matrices. They established laws of large numbers for the extreme
eigenvalues of fixed or slowly growing principal minors. For a Gaussian Wigner matrix with off-diagonal variance one and diagonal variance $a$, their first-order theory covered the regime $0\leq a\leq2$,
including the classical GOE case $a=2$. Since then, only a small number of
works have further developed the theory. \cite{hu2023extreme} extended parts of the Wishart results to non-Gaussian entries under certain moment assumptions. For fixed-order principal minors of the
classical GOE, \cite{feng2026principal}
established Gumbel fluctuations and studied the eigenvector associated
with the maximizing minor. More recently, \cite{jiang2024largest} analyzed the case $k=2$ for deformed GOE and
Wishart matrices and showed that the limiting distribution may change
qualitatively when the diagonal variance or fourth moment crosses a
critical threshold.

The goal of the present article is to propose a new technique for
deriving first-order limits for problems of this type. To give a
snapshot of the framework, let us briefly describe the setup. Given a
random matrix $\bm X\in\mb R^{p\times n}$, define
\begin{equation*}
    T_{n,k}
    :=
    \max_{\substack{
        S\subset[n]\\
        |S|=k
    }}
    \lambda_{\max}
    \left(
        \bm X_{[S]}^\top\bm X_{[S]}
    \right),
\end{equation*}
which is the largest squared singular value among all $p\times k$
column submatrices of $\bm X$. Equivalently, if supports
of size at most $k$ are allowed, it is the maximum of $\|\bm X\bm x\|_2^2$ over
all unit $k$-sparse vectors.

A closely related setting is that of Wigner matrices: given a random
symmetric matrix $\bm A\in\mb R^{n\times n}$, consider
\begin{equation*}
    M_{n,k}
    :=
    \max_{\substack{
        \|\bm x\|_2=1\\
        \|\bm x\|_0\leq k
    }}
    \bm x^\top\bm A\bm x
    =
    \max_{\substack{
        S\subset[n]\\
        |S|\leq k
    }}
    \lambda_{\max}(\bm A_S),
\end{equation*}
where $\bm A_S$ denotes the principal minor indexed by $S$.

We are interested in the first-order limits of $T_{n,k}$ and
$M_{n,k}$: namely, finding sequences of normalizing constants
$\{b_n:n\geq1\}$ and $\{\widetilde b_n:n\geq1\}$ that may depend on
the data distribution such that
\[
    \frac{T_{n,k}}{b_n}
    \stackrel{\mb P}{\to}
    1
    \qquad\text{and}\qquad
    \frac{M_{n,k}}{\widetilde b_n}
    \stackrel{\mb P}{\to}
    1.
\]
The above problem can be quite challenging because the scaling is
non-universal: it depends on the underlying distributions of
$\bm X$ and $\bm A$, as well as on the value of $k$ (especially when
$k$ is fixed), and can be difficult to characterize exactly when the
data distribution is non-Gaussian. This difficulty helps explain why
the existing works \cite{cai2021asymptotic,feng2026principal} focus on
the Gaussian setting. Moreover, it was noted in \cite{cai2021asymptotic} that,
even in the Gaussian Wigner model with off-diagonal variance one and
diagonal variance $a>0$, the authors could not obtain a matching lower
bound for the first-order limit when $a>2$.

We note that non-Gaussian extensions were obtained in
\cite{hu2023extreme}; however, the authors did not investigate the
critical regime considered here. After translating their assumptions
into our notation, their results correspond to the regime
$p/\log n\to\infty$ in Section~\ref{sec:wishart}, in which the
first-order limit is universal. In contrast, our results cover the
full range
\[
    \frac{p}{\log n}\to\beta\in[0,\infty],
\]
including both endpoint regimes. When \(\beta<\infty\), the limiting
behavior is generally non-universal and depends on the entry
distribution. For \(\beta\in(0,\infty)\), we characterize the limit
through an entropy-constrained deterministic convex set and obtain an
explicit expression when the entries are standard Gaussian.

To deal with these difficulties, we propose a technique that can both
predict and rigorously prove the first-order limits of these
quantities when $k$ is fixed. The technique transforms the problem
into finding the high-probability convex closure of a random set in
the space of symmetric matrices. Such an approach is especially useful
when $k$ is kept fixed. The first-order limit can then be obtained by
solving a convex optimization problem, which admits several interesting
explicit formulae when the densities are sufficiently regular.

A complete summary of the first-order limits is provided in
Tables~\ref{tab:Wigner-first-order-limits}
and~\ref{tab:Wishart-first-order-limits}. Our main contributions can be summarized as follows.
\begin{itemize}
    \item We introduce a technique for deriving first-order
    limits of extreme spectral functionals over a family
    of principal minors. The idea is to study the random sets formed by the normalized principal minors and identify the deterministic set to which they converge in Hausdorff distance. One then uses the continuous mapping theorem and solves a convex optimization problem to deduce the first-order limit.

    \item We explain the framework by first focusing on Gaussian Wigner matrices. Assuming that the off-diagonal variance is one and
    the diagonal variance is $a>0$, we determine the first-order limit of $M_{n,k}$ for every fixed $k$ and every $a>0$; see Section \ref{sec:Wigner-fixed-k} below. In particular, we
    resolve the open case $a>2$ left by \cite{cai2021asymptotic}. We show
    that the limiting constant exhibits a phase transition at $a=2$:
    \[
        \gamma_k(a)
        =
        \sqrt{a+2(k-1)}
        \qquad
        \text{when }~~0<a\leq2,
    \]
    whereas, when $a>2$, it is determined by the recursion
    \[
        \gamma_k(a)
        =
        \gamma_{k-1}(a)
        +
        \frac{2}{
            \gamma_{k-1}(a)
            +
            \sqrt{\gamma_{k-1}(a)^2+4-a}
        }, \quad k\ge 2, 
    \]
    with $\gamma_1(a)=\sqrt{a} $. 
    Our proof reveals an interesting phenomenon regarding the structure of the extremal principal minors when $a>2$: the deterministic extremal minors form a nested chain as $k$ increases, whereas this phenomenon does not occur when $a\leq 2$; see Proposition \ref{prop:complete-tight-chain} for the precise statement. Results for the regime in which $k$ grows proportionally are also provided (Theorems \ref{thm:c=0} and \ref{c>0}).


    \item For Wishart matrices, we treat a general sub-Gaussian entry
    distribution $\mu$ in the regime
    \[
        \frac{p}{\log n}\to\beta\in[0,\infty].
    \]
    For $\beta \in (0,\infty)$, we prove that the random set of normalized sub-Gram matrices
    converges in Hausdorff distance to a deterministic convex set
    $\Gamma_{k,\beta,\mu}$ that can be expressed by certain constraints involving the Kullback–Leibler (KL) divergence; see Theorem \ref{thm:Wishart-limit} for the precise statement.
    As a direct consequence, we derive the first-order limit by solving the optimization problem explicitly for the normal distribution: the limit involves
    the unique $\lambda>1$ satisfying
    \[
        \lambda-\log\lambda-1
        =
        \frac{2k}{\beta}.
    \]
    We also characterize the two endpoint regimes $\beta=0$ and $\beta=\infty$. When
    $p/\log n\to\infty$, the first-order limit is universal, which agrees with \cite{cai2021asymptotic}, 
    and when $p/\log n\to0$, the limit is governed by the
    tail of $\mu$.
\end{itemize}

The rest of the paper is organized as follows. The main results for
Wigner and Wishart matrices are presented in Sections~\ref{sec:Wigner}
and~\ref{sec:wishart}, respectively. An outline of the proof technique
and an application to compressed sensing are given in
Section~\ref{sec:app}. Further discussion and remarks are provided in
Section~\ref{sec:discussion}. The remainder of the paper is devoted to
the proofs and other technical results.

\subsection*{Notation} Throughout the paper, the notation $A_n \lesssim B_n$ indicates that $A_n \leq C B_n$ for some universal constant $C>0$. We also write $A_n \lesssim_{a,b} B_n$ to indicate that the constant $C$ may depend on $a$ and $b$ but is independent of $n$. For two matrices $\bm A$ and $\bm B$, the notation $\bm A \succeq \bm B$ is understood in the positive-semidefinite sense. We write $[n]$ to denote $\la 1,2,\dots,n \ra$. Unless stated otherwise, the distance $d$ used throughout the paper is the metric induced by the Frobenius norm.

\begin{table}[t]
\centering
\footnotesize
\setlength{\tabcolsep}{4pt}
\renewcommand{\arraystretch}{1.50}
\setlength{\extrarowheight}{2pt}
\caption{Gaussian Wigner matrices with off-diagonal variance \(1\)
and diagonal variance \(a>0\).}
\label{tab:Wigner-first-order-limits}

\renewcommand{\tabularxcolumn}[1]{%
    >{\centering\arraybackslash}m{#1}%
}

\begin{tabularx}{\textwidth}{|c|X|}
\hline
\textsf{Regime}
&
\textsf{First-order limit}
\\
\hline

\rule[-1.20em]{0pt}{3.7em}%
Fixed \(k\), \(0<a\leq2\)
&
\(\displaystyle
    \frac{M_{n,k}}{\sqrt{2\log n}}
    \stackrel{\mb P}{\to}
    \sqrt{a+2(k-1)}
\).
\\
\hline

\rule[-1.20em]{0pt}{3.7em}%
Fixed \(k\), \(a>2\)
&
\(\displaystyle
    \frac{M_{n,k}}{\sqrt{2\log n}}
    \stackrel{\mb P}{\to}
    \gamma_k(a)
\),
where 
\[
\gamma_1(a)=\sqrt a, \qquad 
    \gamma_m(a)
    =
    \gamma_{m-1}(a)
    +
    \frac{2}{
        \gamma_{m-1}(a)
        +
        \sqrt{\gamma_{m-1}(a)^2+4-a}
    }.
\]
\\
\hline

\rule[-1.20em]{0pt}{3.7em}%
\(k\to\infty\), \(k/n\to0\)
&
\(\displaystyle
    \frac{M_{n,k}}
    {2\sqrt{k\log(n/k)}}
    \stackrel{\mb P}{\to}
    1
\).
\\
\hline

\rule[-1.20em]{0pt}{3.7em}%
\(k/n\to c\in(0,1]\)
&
\(\displaystyle
    \frac{M_{n,k}}{\sqrt n}
    \stackrel{\mb P}{\to}
    \mc E(c)
\),
where \(\mc E\) is increasing and concave on \((0,1)\).
\\
\hline

\end{tabularx}
\end{table}

\begin{table}[t]
\centering
\footnotesize
\setlength{\tabcolsep}{4pt}
\renewcommand{\arraystretch}{1.50}
\setlength{\extrarowheight}{2pt}
\caption{Wishart matrices with sub-Gaussian distributions.}
\label{tab:Wishart-first-order-limits}

\begin{tabularx}{\textwidth}{|c|X|}
\hline
\textsf{Regime}
&
\textsf{First-order limit}
\\
\hline

\rule[-1.15em]{0pt}{3.5em}%
\(\displaystyle \frac{p}{\log n}\to\infty\)
&
\(\displaystyle
    \frac{T_{n,k}}{p}
    \stackrel{\mb P}{\to}
    1.
\)
\\
\hline

\rule[-1.15em]{0pt}{3.5em}%
\(\displaystyle
    \frac{p}{\log n}
    \to\beta\in(0,\infty)
\)
&
\(\centering{\displaystyle
    \frac{T_{n,k}}{p}
    \stackrel{\mb P}{\to}
    \max_{\bm Q\in\Gamma_{k,\beta,\mu}}
    \lambda_{\max}(\bm Q)}
\).
\\
\hline
\rule[-1.15em]{0pt}{3.5em}%
\(\displaystyle
    \frac{p}{\log n}
    \to\beta\in(0,\infty), 
    \text{Gaussian data}
\)

&
\(\centering{\displaystyle
    \frac{T_{n,k}}{p}
    \stackrel{\mb P}{\to}
    \lambda_{k,\beta}
    }
\),
where $\lambda_{k,\beta}>1$ solves 
$\lambda - \log \lambda -1 =  (2k)/\beta$.
\\
\hline

\rule[-1.15em]{0pt}{3.5em}%
\(\displaystyle \frac{p}{\log n}\to0\)
&
\(\displaystyle
    \frac{T_{n,k}}{k\log n}
    \stackrel{\mb P}{\to}
    \kappa_\mu^{-1}
\),
where
$
    \kappa_\mu
    :=
    \lim_{x\to\infty}
    \frac{-\log\mb P(\xi^2>x)}{x}
    \in(0,\infty).
$
\\
\hline

\end{tabularx}
\end{table}

\section{Gaussian Wigner matrices} \label{sec:Wigner}

\subsection{Results for fixed $k$} \label{sec:Wigner-fixed-k}

Fix an integer $k\ge 1$ and a constant $\alpha>0$. For every $n\ge k$, let
$\bm A=(A_{ij})_{1\le i,j\le n}$ be a real symmetric random matrix such that
\[
 A_{ij}=A_{ji}\sim N(0,1),\qquad 1\le i<j\le n,
\]
\[
 A_{ii}\sim N(0,\alpha^2),\qquad 1\le i\le n,
\]
and all random variables on and above the diagonal are independent. Define
\[
 M_{n,k}
 :=\max_{\substack{\|\bm x\|_0\le k\\ \|\bm x\|_2=1}} \bm x^\top \bm A \bm x
 =\max_{\substack{S\subset[n]\\ |S|\le k}}\lmax(\bm A_S),
\]
where $A_S$ is the principal minor indexed by $S$. Put
\[
 q_n:=\sqrt{2\log n},\qquad a:=\alpha^2.
\]
To describe the limit, let us define the compact, convex set
\begin{equation}\label{eq:feasible-set}
 \mathcal K_m(a)
 := \left\{\bm B= \bm B^\top\in\R^{m\times m}:
 \frac1a\sum_{i\in U}B_{ii}^2
   +\sum_{\substack{i<j\\ i,j\in U}}B_{ij}^2 \le |U|\text{ for every }\varnothing\ne U\subset[m]
 \right\}.
\end{equation}
Put
\begin{equation}\label{eq:cost}
 \mathcal I_{U,a}(\bm B)
 := \frac1a\sum_{i\in U}B_{ii}^2
   +\sum_{\substack{i<j\\ i,j\in U}}B_{ij}^2.
\end{equation}
Roughly speaking, $\mc K_m$ consists of all symmetric matrices of size $m$ all of whose principal minors satisfy constraints of the form $\mc I_{U,a} \lb \bm B \rb \leq |U|$. In what follows, we simply write $\mc I_U(\bm B)$ when there is no ambiguity about the variance value $a$.

Set
\begin{equation}\label{eq:gamma-def}
 \gamma_m(a):=\max_{\bm B\in\mathcal K_m(a)}\lmax(\bm B).
\end{equation}
Since the set $\mathcal K_m(a)$ is compact, the maximum exists and is finite. When $k$ is fixed, our first-order limit is as follows.

\begin{thm}\label{thm:main}
For every fixed $k\ge1$ and $\alpha>0$,
\begin{equation}\label{eq:LLN}
 \frac{M_{n,k}}{\sqrt{2\log n}}
 \stackrel{\mb P}{\to}
 \gamma_k(\alpha^2).
\end{equation}



\end{thm}
Theorem \ref{thm:main} gives a first-order limit for every value of the variance $\alpha$. The function $\gamma_k(.)$ can be computed explicitly by solving a convex optimization problem. Interestingly, it exhibits a phase transition at $\alpha=\sqrt{2}$: when $\alpha \leq \sqrt{2}$, $\gamma_k(a)$ has a simple explicit expression that matches \cite{cai2021asymptotic}, whereas it is defined recursively when $\alpha>\sqrt{2}$. Moreover, in the latter case, the maximizers have a hierarchical structure: the matrices that achieve the maximum at different values of $k$ form a nested chain.

\begin{thm} \label{thm:combinatorics}
    The following statements hold:
    \begin{enumerate}[label=\textup{(\roman*)}]
\item If $0<a\le2$, then
\begin{equation}\label{eq:subcritical-constant}
 \gamma_k(a)=\sqrt{a+2(k-1)}.
\end{equation}

\item If $a>2$, then $\gamma_1(a)=\sqrt a$ and, recursively for $k\ge2$,
\begin{align} 
 \gamma_k(a)
& =\max_{0\le t\le1}
 \left\{
 t \gamma_{k-1}(a)+\sqrt{(1-t)\bigl(a(1-t)+4t\bigr)}
 \right\} \label{eq:variational-recursion} \\
 &= \gamma_{k-1}(a) +\frac{2}{\gamma_{k-1}(a)+\sqrt{\gamma_{k-1}(a)^2+4-a}}. \label{eq:closed-recursion}
\end{align}
\end{enumerate}
\end{thm}
Consider the case $k=2$. In this case, Theorem \ref{thm:combinatorics} gives
\begin{equation}\label{eq:k2-constant}
 \gamma_2(\alpha^2)
 =\begin{cases}
 \sqrt{\alpha^2+2},&0<\alpha\le\sqrt2,\\[1ex]
 \displaystyle \alpha+\frac{2}{\alpha+2},&\alpha>\sqrt2.
 \end{cases}
\end{equation}
This agrees with \cite{jiang2024largest}, which studies the limiting
distribution of \(M_{n,2}\). For \(a\in(0,2]\), the first-order limit \(\gamma_k(a)\) also agrees with the earlier results of \cite{cai2021asymptotic}. Theorem~\ref{thm:combinatorics} sheds light on
why deriving the limiting distribution of \(M_{n,k}\) is substantially more difficult when \(a>2\) than when \(a\in(0,2]\). Indeed, in the proof of Theorem~\ref{thm:combinatorics}, we also establish several properties of the maximizers. When \(a\in(0,2]\), a maximizer at level \(k\) can be constructed explicitly and independently of the maximizers at levels \(k-1\) and below. In contrast, when \(a>2\), the maximizers at different levels form a nested chain; see Proposition \ref{prop:complete-tight-chain} below for a precise statement. This is consistent with the limiting
distribution derived in \cite{jiang2024largest}, according to which \(M_{n,2}\) may fail to have a Gumbel limit when \(a>2\).

We will not prove Theorem \ref{thm:main} directly, but instead derive it as a consequence of the following abstract first-order limit for random sets in the space of symmetric matrices; see Theorem \ref{thm:master} below. Before presenting this result, we describe the setup.
Suppose $(X,d)$ is a metric space, and let $A, B \subset X$ be two nonempty subsets. The Hausdorff distance $d_{\rm H}(A,B)$ is defined by
\[
d_{\rm H}(A,B):= \max \la \sup_{x \in A} d \lb x, B \rb; \sup_{y \in B} d(A,y) \ra
\]
where
\[
d(x,B):= \inf_{y \in B} d(x,y), \qquad d(A,y):= \inf_{x \in A} d(x,y).
\]
It is easy to check that $d_{\rm H}(A,B)=0$ if and only if $A$ and $B$ have the same closure. 
Our abstract convergence result roughly states that if we scale all entries of $\bm A$ by a factor of order $1/\sqrt{\log n}$, then the random sets formed by the rescaled principal minors converge to a deterministic set in the Hausdorff distance $d_{\rm H}$.

In our context, $X$ is taken to be the space of all symmetric matrices of size $k$:
\[
X:= \la \bm B\in\R^{k \times k}: \bm B= \bm B^\top  \ra.
\]
The metric $d$ is induced by the Frobenius norm. Thus, $\mc K_k(a)$ is a compact, convex subset of $X$.

\begin{thm} \label{thm:master}
Put
\[
\mathcal{C}_{n,k}:= \la \frac{\bm A_S}{q_n}; S \subset [n] \, \text{and} \, |S|=k \ra_{}.
\]
Then
\[
d_{\rm H} \lb \mc C_{n,k}, \mc K_{k}(a) \rb \stackrel{\mb P}{\to} 0.
\]
\end{thm}
Theorem \ref{thm:master} states that the random set formed by all principal minors of $\bm A$, after scaling by $q_n$, becomes dense in $\mc K_k(a)$ with respect to the Hausdorff distance. Given Theorem \ref{thm:master}, since the functional $\bm A \mapsto \lambda_{\rm max} \lb \bm A \rb$ is convex and $1$-Lipschitz, Lemma \ref{convexity} gives
\begin{align}
    \left| \frac{M_{n,k}}{q_n} - \gamma_k(a) \right| = \left| \max_{\bm B \in \mc C_{n,k}} \lambda_{\rm max} \lb \bm B \rb - \max_{B \in \mc K_{k}(a)} \lambda_{\rm max} \lb \bm B \rb \right| 
    &\leq d_{\rm H} \lb \mbox{conv} \lb \mc C_{n,k} \rb, \mc K_k(a)  \rb \nonumber \\ 
    &\leq d_{\rm H} \lb \mc C_{n,k}, \mc K_{k}(a) \rb \stackrel{\mb P}{\to} 0. \label{continuous-mapping}
\end{align}
Here the first inequality follows from Lemma \ref{convexity} and the fact that $\mc K_k(a)$ is convex, while the second follows because convexifying compact sets does not increase the Hausdorff distance (see, for example, Proposition 1.4 in \cite{iusem2010distances}). This yields Theorem \ref{thm:main}.

We note that the Gaussian assumption is crucial for Theorem \ref{thm:combinatorics}; that is, the optimization problem admits an elegant explicit formula for $\gamma_m(.)$ because of the regular behavior of the Gaussian density. A universality-type result should not be expected, since the marginal distributions of the largest eigenvalues of the principal minors differ across values of $k$. Interestingly, even for growing $k$, when the largest eigenvalue of each principal minor is asymptotically Tracy--Widom distributed, universality of the first-order limit can still fail.

\subsection{Results for growing $k$}

We now investigate the first-order limit of $M_{n,k}$ when $k=k_n \to \infty$. The scaling of $M_{n,k}$ is quite different from that in the fixed-$k$ case. Assume $k/n \to c \in [0,1]$. We will show that when $c=0$, the logarithmic scaling of extreme-value statistics persists, whereas when $c>0$, we have $M_{n,k} \sim \mc E(c) \sqrt{n}$ for some function $\mc E(c)$. The latter scaling is the same as that of the largest eigenvalue of Wigner matrices and suggests that the asymptotic distribution might interpolate between the Tracy--Widom and Gumbel laws.

\begin{thm} \label{thm:c=0}
    Suppose $k \to \infty$ and $k/n \to 0$. Then
    \[
    \frac{M_{n,k}}{2\sqrt{k \cdot \log \lb n/k \rb}} \stackrel{\mb P}{\to } 1.
    \]
\end{thm}
Theorem \ref{thm:c=0} states that when $k \to \infty$ and $k/n \to 0$, the first-order limit of $M_{n,k}$ is independent of $a$. We adapt the second-moment method to prove Theorem \ref{thm:c=0}, using a sparse set of points on the sphere. At first glance, one might guess that when $k/n \to c>0$, the corresponding first-order limit of $M_{n,k}$ would be $2\sqrt{c \log (1/c)} \cdot \sqrt{n}$. However, this is not the case (although the scaling is indeed $\sqrt{n}$), as we will see below.

\begin{thm} \label{c>0}
    Suppose $k/n \to c \in (0,1]$. Then the following statements hold:
    \begin{itemize}
        \item There exists a function $\mc E: (0,1] \to (0,2]$ such that 
        \[
    \frac{M_{n,k}}{\sqrt{n}} \stackrel{\mb P}{\to} \mc E(c). 
    \]
    \item The function $\mc E$ is increasing and concave on $(0,1]$, and $\mc E(1)=2$.
    \item As $c \to 0$, we have 
    \[
    \lim_{c \to 0} \frac{\mc E(c)}{2\sqrt{c\cdot \log \lb 1/c \rb}} = 1.
    \]
    \item As $c \to 1$, we have 
        \[
        \frac{\pi}{6}\leq
        \liminf_{c\uparrow1}\frac{2-\mc E(c)}{(1-c)^3}
        \leq
        \limsup_{c\uparrow1}\frac{2-\mc E(c)}{(1-c)^3}
        \leq \frac{\pi}{3}.
        \]
    \end{itemize}
\end{thm}
We are currently not aware of any explicit expression for the function
$\mc E(c)$. However, it is easy to check that
$\mc E(c)\not\equiv 2\sqrt{c\log(1/c)}$: near $c=1$, $\mc E(c)$ is
close to two, whereas the latter function is close to zero. It is evident
that $\mc E$ can be extended continuously to zero by setting
$\mc E(0)=0$. Interestingly, the third item of
Theorem~\ref{c>0} shows that $\mc E$ is not differentiable at $0$,
whereas the fourth item shows that it is differentiable from the left at
$1$, with $\mc E'(1)=0$.


\section{Results for Wishart matrices} \label{sec:wishart}

We now extend the framework developed for Wigner matrices to Wishart matrices. Although the ideas and framework are the same, a few technical difficulties arise because we treat general non-Gaussian distributions. We begin with the setting. Define
\[
\bm X =  \bm X_n:= \la \xi_{ij}: \, 1\leq i \leq p, \, 1\leq j \leq n \ra \in \mb R^{p \times n}
\]
where the entries are i.i.d. copies of a random variable $\xi$ with law $\mu$. We assume that
\begin{align} \label{regularity conds}
    \mb E \xi = 0, \qquad \mb E \xi^2=1, \qquad \mb E \exp \lb t_0 \xi^2 \rb< \infty
\end{align}
for some $t_0>0$.

We are interested in the quantity
\[
T_{n,k} := \max_{|S|=k} \lmax \lb \bm X_{[S]}^\top \bm X_{[S]} \rb. 
\]
As in the Wigner case, we deduce the LLN for $T_{n,k}$ by determining the closure of the random set formed by rescaling all principal minors of size at most $k$. More precisely, for a fixed $k$, we seek the high-probability limit, with respect to $d_{\rm H}$, of the random set
\begin{align}
    \mc C^{\rm Wishart}_{n,k}:= \la \frac{ \bm X_{[S]}^\top \bm X_{[S]} }{p};  S \subset [n] \, \text{and} \, |S|=k \ra.
\end{align}
Our first result is the analog of Theorem \ref{thm:master} for the Wishart case. To describe the limit, let $\mc P_2 \lb \mb R^k \rb$ be the set of all Borel probability measures on $\mb R^k$ with finite second moment, and let $D \lb \mu \| \nu \rb$ denote the KL divergence
\[
D \lb \mu \| \nu \rb:= \begin{cases}
    \intop \log \lb \frac{d \mu}{d \nu} \rb \, d\mu, \qquad &\text{if $\mu \ll \nu$}, \\
    +\infty   &\text{otherwise}.
\end{cases}
\]
Define
\begin{align}
\mc M_{k, \beta, \mu} &:=  \la \nu \in \mc P_2 \lb \mb R^k \rb:   D \lb \nu_U \Big\| \,  \mu^{\otimes |U|} \rb \leq \frac{|U|}{\beta} \, \text{for all} \, \emptyset \neq U \subset [k]   \ra,   \\
\Gamma_{k,\beta, \mu}   &:= \la \int \bm x \bm x^\top \, d\nu:  \nu \in \mc M_{k,\beta, \mu}   \ra.
\end{align}
\begin{thm} \label{thm:Wishart-limit}
Suppose $p/\log n \to \beta \in  \lb 0, \infty \rb$. The following statements hold:
\begin{itemize}
    \item For every $k \geq 1$ and every $\mu$ satisfying \eqref{regularity conds}, $\Gamma_{k,\beta,\mu}$ is a totally bounded and convex set.
    \item As $n \to \infty$, we have
    \[
    d_{\rm H} \lb \mc C^{\rm Wishart}_{n,k}, \Gamma_{k, \beta, \mu}  \rb \stackrel{\mb P}{\to} 0.
    \]
\end{itemize}
    
\end{thm}

As a direct consequence of the second item in Theorem \ref{thm:Wishart-limit}, and arguing as in \eqref{continuous-mapping}, we deduce that
\begin{align*}
    \frac{T_{n,k}}{p} \stackrel{\mb P}{\to} \max_{\bm Q \in \Gamma_{k,\beta, \mu}} \la \lmax \lb \bm Q \rb \ra.
\end{align*}
In general, the convex optimization problem over $\Gamma_{k,\beta,\mu}$ in the last display does not admit a closed-form solution. However, explicit expressions may be available when $\mu$ has a sufficiently regular density. We now consider the case in which $\mu$ is the standard normal distribution.

\begin{co} \label{co:Wishart-Gaussian}
    If $\mu \equiv N(0,1)$, then 
    \[
     \frac{T_{n,k}}{p} \stackrel{\mb P}{\to} \lambda_{k,\beta}
    \]
    where $\lambda_{k,\beta}$ is the unique solution in $(1,\infty)$ of the equation 
    \[
\lambda - \log \lambda -1 =  \frac{2k}{\beta}.
\]
\end{co}
Readers may wonder why Theorem \ref{thm:Wishart-limit} does not cover the cases $\beta=0$ and $\beta=\infty$. The LLN for $T_{n,k}$ is different in these two cases and can be derived using simple arguments. If $p/\log n \to \infty$, then
\[
\frac{T_{n,k}}{p} \stackrel{\mb P}{\to} 1.
\]
This convergence is universal: it holds for every law $\mu$ satisfying \eqref{regularity conds}. It can be proved using a simple union bound to show that
\[
\sup_{|S|=k} \left\| \frac{\bm X_{[S]}^\top  \bm X_{[S]}   }{p} - \bm I_n  \right\|_\infty \stackrel{\mb P}{\to} 0.
\]
We next compare our results with Theorem 1 in \cite{cai2021asymptotic}. Written in the same notation, Theorem 1 in \cite{cai2021asymptotic} implies that
\[
\frac{T_{n,k}}{p} \stackrel{\mb P}{\to} 1 \qquad \text{if $\sqrt{p}/ \log n \to \infty$}.
\]
In particular, our results recover the preceding limit under the weaker condition $p/\log n \to \infty$ and give a first-order limit for the more difficult regime in which $p$ grows on the same scale as $\log n$.
In the case $p / \log n \to 0$, the limit depends on the law $\mu$. When the tail behavior of $\mu$ is sufficiently regular, we can show the following.
\begin{thm} \label{p/logn to 0}
    Suppose $p/\log n \to 0$ and 
    \[
    \kappa_\mu:= \lim_{x \to \infty} \frac{-\log \mb P \lb \xi^2>x \rb}{x} \in (0,\infty).
    \]
    Then
    \[
    \frac{T_{n,k}}{k\log n} \stackrel{\mb P}{\to} \frac{1}{\kappa_\mu}.
    \]
\end{thm}
It is not difficult, albeit more technically involved, to modify the proof of Theorem \ref{p/logn to 0} to obtain analogous results under sub-Weibull type conditions. We do not pursue this technical extension in the current paper and leave it as potential future work. 

{
\section{Proof technique and an application} \label{sec:app}

\subsection{Proof technique}

We first use the Gaussian Wigner matrix to illustrate the main idea. We seek to identify the deterministic set to which the random sets formed by the rescaled principal minors converge in Hausdorff distance. Suppose $\bm B \in \mb R^{k \times k}$ is a candidate element of this limiting set. Rescale $\bm A$ by $\bm A/q_n$, where $q_n=\sqrt{2\log n}$, so that all entries of the rescaled matrix are of order $O_{\mb P}(1)$.

For $S\subset [n]$ with $|S|=k$, it is easy to check that
\begin{align} \label{gaussian}
    \mb P\left(
        \frac{\bm A_S}{q_n}\approx\bm B
    \right)
    =
    n^{- \mc I_{S,a}(\bm B)+o(1)},
\end{align}
where
\[
    \mc I_{S,a}(\bm B)
    :=
    \frac1a\sum_{i\in S}B_{ii}^2
    +
    \sum_{\substack{i<j\\i,j\in S}}B_{ij}^2.
\]
There are $\Theta \lb n^k \rb$ such principal minors $\bm A_S$, so
\[
\mb P \lb \exists S: |S|=k \,\, \text{and} \,\, \frac{\bm A_S}{q_n}\approx\bm B \rb \approx n^{k-  \mc I_{S,a}(\bm B)+o(1)}.
\]
For the above probability to be $1+o(1)$, we must have
\[
k \geq  \mc I_{S,a}(\bm B).
\]
However, the preceding constraint is not sufficient: if $\bm A_S/q_n$ is close to $\bm B$, then all principal minors of $\bm A_S$ are also close to the corresponding principal minors of $\bm B$. Take $U \subset [k]$ with $|U|= r \leq k$. Then
\[
\mb P \lb \exists V: |V|=r \,\, \text{and} \,\, \frac{\bm A_V}{q_n}\approx\bm B_U \rb \approx n^{r-  \mc I_{U,a}(\bm B)+o(1)}.
\]
For the probability above to be $1+o(1)$, it is necessary that
\[
    \mc I_{U,a}(\bm B)\leq |U|
    \qquad
    \text{for every }\varnothing\neq U\subseteq[k].
\]
This is precisely the constraint in \eqref{eq:feasible-set}. The preceding heuristic can be made rigorous via the second-moment method. In the proof of the lower bound, we use a blocking construction that substantially simplifies the argument compared with that of \cite{cai2021asymptotic}.

In the Wishart case, we replace the approximation \eqref{gaussian} by Sanov's theorem; see Section 6.2 in \cite{dembo2010large} for more details. Technical complications arise because the KL divergence does not have a closed-form expression analogous to that of the convex functional $\mc I_{U,a}$ above. We handle these difficulties using a combinatorial argument; see the proof of the upper bound in Section \ref{sec:upper-Wishart} below.

\subsection{An application to compressed sensing}
The Hausdorff convergence in Theorem \ref{thm:Wishart-limit} yields a sharp asymptotic threshold for the restricted isometry property of the normalized design matrix $\bm X/\sqrt p$ in compressed sensing. For a fixed order $k$, define
\begin{align} \label{RIP-constant}
    \delta_k \lb \bm X/\sqrt p \rb
    := \max_{\substack{S\subset[n]\\ |S|=k}}
    \left\| \frac{\bm X_{[S]}^\top \bm X_{[S]}}{p}-\bm I_k \right\|_{\op}.
\end{align}
The maximum over supports of size exactly $k$ agrees with the usual maximum over supports of size at most $k$, because any smaller support can be enlarged to a $k$-element set. Equivalently, $\delta_k \lb \bm X/\sqrt p \rb$ is the smallest $\delta\geq 0$ such that
\[
    \lb 1-\delta \rb \|\bm u\|_2^2
    \leq \frac{\|\bm X\bm u\|_2^2}{p}
    \leq \lb 1+\delta \rb \|\bm u\|_2^2
\]
for every $\bm u\in\mb R^n$ satisfying $\|\bm u\|_0\leq k$.
As a consequence of Theorem \ref{thm:Wishart-limit}, we can prove that

\begin{co}[Restricted isometry threshold] \label{co:restricted-isometry}
Let $k\geq 1$ be fixed. Suppose that $p/\log n\to\beta\in\lb 0,\infty\rb$ and that \eqref{regularity conds} holds. Then
\[
    \delta_k \lb \bm X/\sqrt p \rb
    \stackrel{\mb P}{\to}
    \Delta_{k,\beta,\mu}
    :=\sup_{\bm Q\in\Gamma_{k,\beta,\mu}}
    \|\bm Q-\bm I_k\|_{\op}.
\]
If $\mu=N(0,1)$ and $\lambda_{k,\beta}>1$ is the unique solution of
\[
    \lambda_{k,\beta}-\log\lambda_{k,\beta}-1=\frac{2k}{\beta},
\]
then
\[
    \Delta_{k,\beta,\mu}=\lambda_{k,\beta}-1.
\]
Consequently, for every fixed $\delta\in(0,1)$,
\[
\mb P\left(\delta_k\lb\bm X/\sqrt p\rb\leq\delta\right)
\to
\begin{cases}
1, & \displaystyle \beta>\frac{2k}{\delta-\log\lb 1+\delta\rb},\\[2mm]
0, & \displaystyle \beta<\frac{2k}{\delta-\log\lb 1+\delta\rb}.
\end{cases}
\]
\end{co}

We now discuss the implications of Corollary \ref{co:restricted-isometry} for the restricted isometry threshold. The matrix $\bm X/\sqrt p$ may be viewed as a sensing matrix with $p$ measurements in ambient dimension $n$. Fix an integer $s\geq1$. A deterministic theorem of \citet{cai2014sparse} shows that the condition $\delta_{2s}<1/\sqrt2$ guarantees uniform exact recovery of an $s$-sparse signal by basis pursuit. Applying Corollary \ref{co:restricted-isometry} with $k=2s$, we define
\[
    \beta_{\rm BP}(s)
    :=\frac{4s}{2^{-1/2}-\log\lb1+2^{-1/2}\rb}
    \approx 23.2144\,s.
\]
If $\mu=N(0,1)$ and $p/\log n\to\beta$ for some $\beta>\beta_{\rm BP}(s)$, then $\delta_{2s}\lb\bm X/\sqrt p\rb<1/\sqrt2$ with probability tending to one. Consequently, a logarithmic number of Gaussian measurements suffices for uniform basis-pursuit recovery at any fixed sparsity level, with an explicit sufficient leading constant.
}

\section{Discussion and remarks} \label{sec:discussion}

In this paper, we have developed a new technique for establishing the
first-order limits of the largest eigenvalues among all principal minors
of Wishart matrices and Gaussian Wigner matrices. Our technique exploits
the convergence, in Hausdorff distance, of the random sets formed by the
principal minors. We expect this technique to apply beyond the two types
of random matrices considered here. We now make
several remarks regarding possible extensions for future research.

\begin{itemize}
    \item Although we show that the limit $\mc E(c)$ exists in
    Theorem~\ref{c>0}, its properties are not well understood apart from its
    differentiability at $1$ and non-differentiability at $0$. We expect
    tools from spin glass theory \cite{chen2013aizenman,chen2017parisiSpherical} to be useful for understanding the
    behavior of $\mc E(c)$.

    \item The regime $k/n\to c\in(0,1]$ for Gaussian Wishart matrices
    seems to be completely different from that for Gaussian Wigner
    matrices. In particular, a direct Gaussian comparison argument does
    not work. It would be interesting to prove and characterize the
    first-order limit in this case.

    \item Studying the limiting distributions of $T_{n,k}$ and
    $M_{n,k}$ is an important direction for future research. Our results
    for $M_{n,k}$ partially explain why establishing the limiting
    distribution for Gaussian Wigner matrices is much more difficult
    when the diagonal variance is greater than two: the nested structure
    implies that the extremal principal minors of different sizes form a
    nested chain, and the standard Poisson approximation argument, as in
    \cite{arratia1989two,feng2026principal}, does not capture such a
    dependence structure.
\end{itemize}

\section{Proof of Theorem \ref{thm:master}}

To prove Theorem \ref{thm:master}, it suffices to prove that
\begin{align} \label{upper}
 \sup_{\bm C \in \mc C_{n,k}} d \lb \bm C, \mc K_k(a) \rb =  \sup_{\bm C \in \mc C_{n,k}} \inf_{\bm H \in \mc K_{k}(a)} \| \bm C - \bm H \|_F \stackrel{\mb P}{\to} 0
\end{align}
and
\begin{align} \label{lower}
   \sup_{\bm H \in \mc K_k(a)} d \lb \mc C_{n,k}, \bm H \rb = \sup_{\bm H \in \mc K_{k}(a)} \inf_{\bm C \in \mc C_{n,k}} \| \bm C - \bm H \|_F \stackrel{\mb P}{\to} 0.
\end{align}

\noindent \underline{\it Proof of \eqref{upper}.} We first show that for every $\ve >0$, 
\begin{equation}\label{eq:uniform-feasibility}
 \Pp\left(
 \mathcal I_{T,a}\!\left(\frac{A_T}{q_n}\right)
 \le |T|+\varepsilon
 \text{ for all }T\subset[n]\text{ with }1\le|T|\le k
 \right)\to 1.
\end{equation}
To prove \eqref{eq:uniform-feasibility}, fix $s\le k$ and a set $T\subset[n]$ with $|T|=s$. Then
\[
 q_n^2\cdot \mathcal I_{T,a}\!\left(\frac{A_T}{q_n}\right)
 =\frac1a\sum_{i\in T}A_{ii}^2
  +\sum_{\substack{i<j\\i,j\in T}}A_{ij}^2
 \sim\chi^2_{d_s},
 \qquad d_s=\frac{s(s+1)}2.
\]
For a fixed $d$, the chi-square tail satisfies
\[
 \Pp(\chi_d^2>x)\le C_d x^{d/2}e^{-x/2}
\]
for all sufficiently large $x$.

Since $q_n^2=2\log n$, we have 
\[
 \Pp\left(
 \mathcal I_{T,a}\!\left(\frac{A_T}{q_n}\right)>s+\varepsilon
 \right)
 \le C_s \cdot (\log n)^{d_s} \cdot n^{-(s+\varepsilon)}
\]
where
\[
C_s:= 2^{s(s+1)/2}\cdot (s+\ve)^{s(s+1)/2}.
\]
There are at most $n^s$ choices of $T$ for which $|T|=s$, so a union bound gives
\begin{align*}
    \mb P \lb \exists T: I_{T,a}\!\left(\frac{A_T}{q_n}\right)
 \geq |T|+\varepsilon   \rb 
 \leq \sum_{s=1}^k n^s\cdot  C_s \cdot (\log n)^{d_s} \cdot n^{-(s+\varepsilon)} \to 0.
\end{align*}
This proves \eqref{eq:uniform-feasibility}. 

Now, with $\mc I_{U,a}$ as in \eqref{eq:cost}, define
\[
\mc K^{\delta}_{k}(a):= \la \bm B= \bm B^\top\in\R^{k \times k}:
 \mathcal I_{U,a}(\bm B)\le (1+\delta) \cdot |U|\text{ for every }\varnothing\ne U\subset[k]
 \ra.
\]

By \eqref{eq:uniform-feasibility}, for every fixed $\delta>0$,
$
    \mb P\left(
        \mc C_{n,k}\subseteq\mc K_k^\delta(a)
    \right)\to1.
$
Hence
\begin{align*}
\mb P\left(
    \sup_{\bm C\in\mc C_{n,k}}
    d\left(\bm C,\mc K_k(a)\right)>\ve
\right)
\leq
\mb P\left(
    \mc C_{n,k}\not\subseteq\mc K_k^\delta(a)
\right)
+
\mathbf 1\left\{
    \sup_{\bm B\in\mc K_k^\delta(a)}
    d\left(\bm B,\mc K_k(a)\right)>\ve
\right\}.
\end{align*}
Since
$
    \mc K_k^\delta(a)
    =
    \sqrt{1+\delta}\,\mc K_k(a)
$
and $\mc K_k(a)$ is compact,
\[
    \sup_{\bm B\in\mc K_k^\delta(a)}
    d\left(\bm B,\mc K_k(a)\right)
    \leq
    \left(\sqrt{1+\delta}-1\right)
    \cdot 
    \sup_{\bm H\in\mc K_k(a)}\|\bm H\|_{\rm F}
    \to0
\]
as $\delta\to0$. Therefore, we get \eqref{upper} by first letting $n \to \infty$ and then letting $\delta \to 0$.

\noindent \underline{ \it Proof of \eqref{lower}.} We first show that if a deterministic matrix $\bm B$ of size $k\times k$ satisfies
\begin{equation}\label{eq:strict-feasibility-B}
    \min_{\varnothing\ne U\subseteq[k]}
    \left\{
        |U|-\mathcal I_U(\bm B)
    \right\}
    \geq \zeta >0
\end{equation}
then, for any fixed $\eta>0$,
\begin{align} \label{lower-1}
    \mb P \lb d \lb \bm B, \mc C_{n,k} \rb > \eta \rb \to 0.
\end{align}
Take $\ve:=\eta/k$. We partition $[n]$ into $k$ consecutive blocks. Specifically, define
\[
m_n:= \lfloor n/k \rfloor, \qquad V_{n,r}:= \la (r-1)m_n+1,\dots, rm_n  \ra, \qquad r=1,2,\dots,k.
\]
and
\begin{align} \label{V-tilde}
\widetilde{V}:= V_{n,1}\times\dots\times V_{n,k}.
\end{align}
For a tuple $\bm i= \lb i_1,\dots,i_k \rb \in \widetilde{V}$, we have $i_1<i_2<\dots<i_k$. For such a tuple $\bm i = \lb i_1,\dots, i_k \rb \in \widetilde{V}$ and a symmetric matrix $\bm B$ of size $k$, define the event
\[
E\lb \bm i, \bm B, \ve \rb:= \bigcap_{r=1}^k \la \left| \frac{\bm A_{i_r i_r}}{q_n} - \bm B_{rr} \right| \leq \ve  \ra \cap \bigcap_{1 \leq r<s\leq k} \la  \left| \frac{\bm A_{i_r i_s}}{q_n} - \bm B_{rs} \right| \leq \ve \ra.
\]
This is the event that the principal minor $\bm A_{\la i_1,\dots,i_k \ra}$ lies within a box of size $\ve$ around a deterministic symmetric matrix $\bm B$. Define the counting random variable
\[
N_n=N_n \lb \bm B, \ve \rb:= \sum_{\bm i \in \widetilde{V} } \mathbf{1}_{ E \lb \bm i, \bm B, \ve \rb}.
\]
To deduce \eqref{lower-1}, it suffices to show that $\mb P \lb N_n =0 \rb \to 0$. Indeed, given this statement, with probability tending to one, we can find a tuple $\bm i \in \widetilde{V}$ such that
\[
\left\| \frac{\bm A_{\la i_1,\dots,i_k \ra}}{q_n} - \bm B \right\|_\infty \leq \ve,
\]
which in turn yields
\[
d \lb \bm B, \mc C_{n,k} \rb \leq \left\| \frac{\bm A_{\la i_1,\dots,i_k \ra}}{q_n} - \bm B \right\|_F \leq \sqrt{k^2\ve^2} = k\ve =  \eta.  
\]
Therefore, we only need to show $\mb P \lb N_n =0 \rb \to 0$. By the Paley--Zygmund inequality,
\[
\mb P \lb N_n=0 \rb \leq \frac{\mbox{Var} \lb N_n \rb}{\lb \mb E N_n \rb^2} =  \frac{\mb E \lb N_n^2 \rb}{\lb \mb E N_n \rb^2}-1 .
\]
To bound the right-hand side of the display above, put 
\begin{align} \label{p_n}
p_n\lb \bm i, \bm B, \ve \rb:= \mb P \left[ E\lb \bm i, \bm B, \ve \rb \right].
\end{align}
Then, for any fixed $\bm i^* \in \widetilde{V}$, we have  
\[
\mb E N_n = \# \la \text{all tuples in $\widetilde{V}$ } \ra \cdot p_n\lb \bm i^*, \bm B, \ve \rb = m_n^k \cdot p_n\lb \bm i^{*}, \bm B, \ve \rb \cdot (1+o(1)).
\]
Thus, by Lemma \ref{Gaussian bound} and \eqref{eq:strict-feasibility-B}, 
\begin{align*}
    \log \mb E \lb N_n \rb &= k \cdot \log \lb m_n \rb + \log p_n\lb \bm i, \bm B, \ve \rb + o(1) \\
    &\geq k\cdot \lb \log n - \log 2k \rb - \mc I_{\la i_1,\dots, i_k \ra} \lb \bm B \rb \cdot \log n - C_{\bm B,k,a,\ve} \sqrt{\log n} + O(1) \\
    &\geq \log n \cdot \left[ k - \mc I_{\la i_1,\dots, i_k \ra} \lb \bm B \rb \right] + O(\sqrt{\log n}) \\
    &\geq \zeta\cdot \log n +  O(\sqrt{\log n})  \to \infty.
\end{align*}
Therefore,
\[
\mb E N_n \to \infty. 
\]
To bound the second-moment term, write
\begin{align*}
\mb E \lb N_n^2 \rb = \mb E N_n + \sum_{\bm i \in \widetilde{V}, \bm j \in \widetilde{V}, \bm i \neq \bm j} \mb P \left[
E\lb \bm i, \bm B, \ve \rb \cap E\lb \bm j, \bm B, \ve \rb
\right]
\end{align*}
Write $E_{\bm i} = E\lb \bm i, \bm B, \ve \rb$ and $E_{\bm j} = E\lb \bm j, \bm B, \ve \rb$ for shorthand. Put
\[
U=U \lb \bm i, \bm j \rb:= \la r \in [n]: i_r = j_r \ra. 
\]
Due to the blocking structure of $\widetilde{V}$, any overlap between $\bm i$ and $\bm j$ must occur at exactly the same coordinate and hence belongs to $U$. We denote by $G_U$ the principal minor formed by $U$. We can decompose the events as follows (see Figure \ref{fig:overlap-factorization} for an illustration):
\[
E_{\bm i} \cap E_{\bm j} = G_U, \qquad E_{\bm i} \cup E_{\bm j} = H_i \cup H_j \cup G_U.
\]
Note that $H_i, H_j, G_U$ are mutually disjoint. Consequently,
\begin{align*}
\mb P \left[ E\lb \bm i, \bm B, \ve \rb \cap E\lb \bm j, \bm B, \ve \rb \right]
= \mb P \lb E_{\bm i} \cap E_{\bm j} \rb
&= \mb P \lb H_{\bm i} \rb \cdot \mb P \lb H_{\bm j} \rb \cdot \mb P \lb G_U \rb \\
&=   \frac{\mb P \lb H_{\bm i} \rb \cdot \mb P \lb G_U \rb \cdot  \mb P \lb H_{\bm j} \rb \cdot  \mb P \lb G_U \rb}{\mb P \lb G_U \rb} \\
&= \frac{p_n \lb \bm i, \bm B, \ve \rb^2}{p_n \lb U(\bm i, \bm j) , \bm B, \ve \rb}.
\end{align*}
Here $\mb P \lb H_{\bm i} \rb$ is the probability that all the constraints are satisfied after restricting to the matrix $H_{\bm i}$.

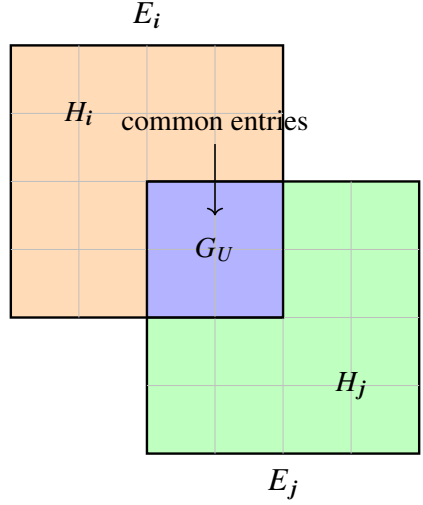
\begin{figure}[t]
\centering
\begin{tikzpicture}[
    x=0.9cm,
    y=0.9cm,
    font=\small,
    every node/.style={inner sep=2pt}
]

\fill[orange!28] (0,0) rectangle (4,4);

\fill[green!25] (2,-2) rectangle (6,2);

\fill[blue!30] (2,0) rectangle (4,2);

\draw[line width=.9pt] (0,0) rectangle (4,4);
\draw[line width=.9pt] (2,-2) rectangle (6,2);

\foreach \t in {1,2,3} {
    \draw[gray!50, line width=.35pt] (\t,0)--(\t,4);
    \draw[gray!50, line width=.35pt] (0,\t)--(4,\t);
    \draw[gray!50, line width=.35pt] (2+\t,-2)--(2+\t,2);
    \draw[gray!50, line width=.35pt] (2,-2+\t)--(6,-2+\t);
}

\draw[line width=.7pt] (2,0) rectangle (4,2);

\node[font=\bfseries] at (2,4.45) {$E_{\bm i}$};
\node[font=\bfseries] at (4,-2.45) {$E_{\bm j}$};

\node at (1.0,3.0) {$H_{\bm i}$};
\node at (3.0,1.0) {$G_U$};
\node at (5.0,-1.0) {$H_{\bm j}$};

\draw[->, line width=.6pt] (3.0,2.55) -- (3.0,1.5);
\node[align=center] at (3.0,2.9) {common entries};





\end{tikzpicture}
\caption{
Decomposition of two overlapping pattern events.
The event \(G_U\) corresponds to the constraints coming from the shared
\(U\times U\) block, while \(H_{\bm i}\) and \(H_{\bm j}\) contain
the remaining constraints in \(E_{\bm i}\) and \(E_{\bm j}\), respectively.
}
\label{fig:overlap-factorization}
\end{figure}

For a fixed $U \subset [k]$, the number of ordered pairs $\lb \bm i, \bm j \rb \in \widetilde{V}^2$ satisfying $U \lb \bm i, \bm j \rb= U$ is no larger than
\[
m_n^{|U|} \cdot m_n^{2(k-|U|)} = m_n^{2k-|U|}.
\]
Consequently, since $\mb E N_n \to \infty$,
\begin{align*}
    \frac{\mb E \lb N_n^2 \rb}{\lb \mb E N_n \rb^2}-1 &\leq  \frac{1}{\mb E N_n} +  \frac{1+o(1)}{m_n^{2k} \cdot  p_n^2\lb \bm i^{*}, \bm B, \ve \rb } \cdot 
   \sum_{l=1}^k \sum_{|U|=l} \frac{p_n^2 \lb \bm i^*, \bm B, \ve \rb}{ p_n \lb U, \bm B, \ve \rb} \\
   &\leq o(1) + \left[ 1 +o(1) \right] \cdot \frac{1}{m_n^{2k}} \cdot \sum_{l=1}^k \frac{m_n^{2k-l}}{ \min_{|U|=l} \,  p_n \lb U, \bm B, \ve \rb } \\
   &\leq  o(1) + \left[ 1 +o(1) \right] \cdot \sum_{l=1}^k \frac{m_n^{-l}}{\min_{|U|=l} \,  p_n \lb U, \bm B, \ve \rb}.
\end{align*}
To bound the last sum, note that for any $U \subset [k]$ with $|U|=l$, by Lemma \ref{Gaussian bound} and \eqref{eq:strict-feasibility-B}, we have 
\begin{align*}
    \frac{m_n^{-l}}{  p_n \lb U, \bm B, \ve \rb} &= \exp \la -l\cdot \log m_n - \log p_n \lb U, \bm B, \ve \rb \ra \\
    &\leq  \exp \la -l\cdot \log n + \mc I_{U} \lb \bm B \rb \cdot \log n + O \lb \sqrt{\log n} \rb   \ra \\
    &\leq \exp \la -\zeta \cdot \log n  + O \lb \sqrt{\log n} \rb \ra \to 0.
\end{align*}
The above bound is independent of $U$, so 
\[
    \frac{\mb E \lb N_n^2 \rb}{\lb \mb E N_n \rb^2}-1 \to 0.
\]
Putting everything together, we conclude that $\mb P \lb N_n =0 \rb \to 0$. To finish the proof, we use a standard density argument to replace $\zeta$ in \eqref{eq:strict-feasibility-B} by zero: any $\bm B  \in \mc K_k(a)$ can be approximated by $\bm B_\delta = (1-\delta) \cdot  \bm B$; note that
\[
\min_{U \subset [k], U \neq \emptyset} \la |U| -  \mc I_U \lb \bm B_\delta \rb \ra = \min_{U \subset [k], U \neq \emptyset} \la |U| - (1-\delta)^2 \mc I_U \lb \bm B \rb  \ra >0.  
\]
The pointwise convergence in probability can be upgraded to uniform convergence by using Theorem 1 in \cite{newey1991uniform} and the fact that the distance $d$ is $1$-Lipschitz. This completes the proof. $\hfill$ $\square$

\section{Proof of Theorem \ref{thm:combinatorics}}
The behavior of $\gamma_m(a)$ differs in the two regimes, which we treat in separate subsections.
\subsection{The case $a \in (0,2]$}

Let $\bm B\in\mathcal K_m(a)$ and let $\bm x\in\R^m$ satisfy $\| \bm x\|_2=1$.  By
Cauchy--Schwarz, we have 
\begin{align*}
 \bm x^\top \bm B \bm x
 &=\sum_{i=1}^m \left(\frac{B_{ii}}{\sqrt a}\right) \cdot (\sqrt a\,x_i^2)
   +\sum_{1\leq i<j \leq m}B_{ij} \cdot (2x_ix_j)\\
 &\le
 \left(\frac1a\sum_{i=1}^m B_{ii}^2+\sum_{1\leq i<j \leq m}B_{ij}^2\right)^{1/2} \cdot 
 \left(a\sum_i x_i^4+4\sum_{i<j}x_i^2x_j^2\right)^{1/2} \\ &
 = \sqrt{\mc I_{[m]}\lb \bm B \rb} \cdot \left(a\sum_i x_i^4+4\sum_{i<j}x_i^2x_j^2\right)^{1/2}.
\end{align*}
The first factor is at most $\sqrt m$ because $\bm B\in\mathcal K_m(a)$.  Moreover, since
\[
 4\sum_{i<j}x_i^2x_j^2
 =2\left(1-\sum_i x_i^4\right),
\]
the square of the second factor equals $2+(a-2)\sum_i x_i^4$.

When $a\le2$, the inequality $\sum_i x_i^4\ge1/m$ gives
\[
 2+(a-2)\sum_i x_i^4
 \le2+\frac{a-2}{m}.
\]
Thus
\[
 \lmax(\bm B)\le\sqrt{a+2(m-1)}.
\]
It suffices to construct a matrix $\bm B_1$ that achieves equality. Let $c=\sqrt{a+2(m-1)}$ and define a matrix $\bm B_1$ such that
\begin{equation}\label{eq:balanced-matrix}
 \bm B_{1}[i,i]=\frac{a}{c},
 \qquad
 \bm B_{1}[i,j]=\frac2c\quad(i\ne j).
\end{equation}
If $|U|=s$, then
\[
 \mathcal I_U(\bm B_1)
 =\frac{s\left[ a+2(s-1)\right]}{a+2(m-1)}\le s,
\]
so $\bm B_1\in\mathcal K_m(a)$.  

Moreover, since every row sum equals $c$, $\lmax(\bm B_1)=c=\sqrt{a+2(m-1)}$. $\hfill$ $\square$

\subsection{The case $a>2$}

The identity $\gamma_1(a)=\sqrt a$ follows immediately from the
constraint $\mathcal I_{\{1\}}(B)\le1$. We prove the recursion for
$m\ge2$. The conclusion of Theorem \ref{thm:combinatorics}, together with additional geometric properties of the maximizers, follows from Propositions \ref{prop:gamma-recursion} and \ref{prop:complete-tight-chain} below.

\begin{lemma}\label{lem:gamma-strict}
For every \(a>0\) and \(m\ge2\),
\[
    \gamma_m(a)>\gamma_{m-1}(a).
\]
\end{lemma}

\noindent \textbf{Proof of Lemma \ref{lem:gamma-strict}.}
Let \( \bm C\in\mathcal K_{m-1}(a)\) attain \(\gamma_{m-1}(a)\), and let
\(\bm u\in\mathbb R^{m-1}\) be a unit top eigenvector of \(\bm C\). For
\(0<\delta\le1\), define
\[
    \widetilde {\bm C}
    :=
    \begin{pmatrix}
        \bm C & \delta \bm u\\
        \delta \bm u^\top & 0
    \end{pmatrix}.
\]
We claim that \(\widetilde{\bm C}\in\mathcal K_m(a)\). Indeed, if \(T\subset[m-1]\), then
\[
\mc I_{T} \lb \widetilde{\bm C} \rb = \mc I_{T} \lb \bm C \rb  \leq |T|
\]
and 
\[
\begin{aligned}
    \mathcal I_{T\cup\{m\}}\lb \widetilde{\bm C} \rb
    &=
    \mathcal I_T(\bm C)
    +
    \delta^2\sum_{i\in T}u_i^2
    \le
    |T|+\delta^2
    \le
    |T|+1.
\end{aligned}
\]
Therefore, for $\bm w = \lb 1/\sqrt{1+\eta^2} \rb \cdot  \lb \bm u, \eta \rb^\top$, where $\eta>0$, we have
\[
    \gamma_m(a)
    \ge
    \lambda_{\max}\lb \widetilde{\bm C} \rb
    \geq \bm w^\top \,  \widetilde{\bm C} \, \bm w =  \frac{\gamma_{m-1}(a)+2\eta\delta}{1+\eta^2}
    >
    \gamma_{m-1}(a)
\]
whenever $\eta$ is chosen sufficiently small that $\eta<2\delta/\gamma_{m-1}(a)$.
$\hfill$ $\square$

\begin{lemma}\label{lem:structure}
Assume that \(a>2\) and \(m\ge2\). If
\(\bm B\in\mathcal K_m(a)\) attains \(\gamma_m(a)\), then
$\mc I_{[m]} \lb \bm B \rb = m$.
\end{lemma}

\noindent \textbf{Proof of Lemma \ref{lem:structure}.}
Let \(\bm v\) be a unit top eigenvector of \(\bm B\). We split the proof into two steps. 
\medskip
\noindent
\underline{ \it Step 1: Up to permutation and diagonal conjugation, \(\bm v\) and \(\bm B\) are entrywise
positive.}
We first show that $\bm v$ has full support. Suppose $S:=\{i\in[m]:v_i\ne0\}$ is a proper subset of \([m]\). Since the principal minor \(B_S\) belongs to \(\mathcal K_{|S|}(a)\), we have
\[
\begin{aligned}
    \gamma_m(a) &= v^\top Bv =v_S^\top B_Sv_S \le \lambda_{\max}(B_S)
    \le \gamma_{m-1}(a),
\end{aligned}
\]
contradicting Lemma~\ref{lem:gamma-strict}. Hence, $S=[m]$.

Now define 
\[
   \bm D
    :=
    \operatorname{diag}
    \bigl(
        \operatorname{sgn}(v_1),\ldots,
        \operatorname{sgn}(v_m)
    \bigr),
\]
and put
\[
    \bm C:=\bm D \bm B \bm D,
    \qquad
    \bm u:=  \bm D \bm v.
\]
Since diagonal sign conjugation preserves the eigenvalues and all the
constraints defining \(\mathcal K_m(a)\),
\[
   \bm C \in\mathcal K_m(a),
    \qquad \text{and} \qquad 
    \lambda_{\max}( \bm C)=\gamma_m(a),
    \qquad
    u_i>0.
\]
Let \(|\bm C|\) denote the entrywise absolute-value matrix obtained from $\bm C$. Since the
constraints \(\mathcal I_U\) depend only on squares of the entries,
\[
    |\bm C|\in\mathcal K_m(a).
\]
Consequently,
\[
    \gamma_m(a)
    =
    \bm u^\top \bm C \bm u
    \le
    \bm u^\top|\bm C| \bm u
    \le
    \lambda_{\max}(|\bm C|)
    \le
    \gamma_m(a).
\]
Thus, equality holds throughout. In particular,
\[
\begin{aligned}
0
&=
\bm u^\top|\bm C| \bm u- \bm u^\top \bm C \bm u
=
\sum_{i=1}^m
u_i^2\bigl(|C_{ii}|-C_{ii}\bigr)
+
2\sum_{1\le i<j\le m}
u_iu_j\bigl(|C_{ij}|-C_{ij}\bigr).
\end{aligned}
\]
Every term in the last expression is nonnegative, and all coefficients
\(u_i^2\) and \(u_iu_j\) are strictly positive. Hence
\[
    C_{ij}\ge0
    \qquad\text{for all }i,j.
\]
Replacing \(\bm B\) by \(\bm C\) and \(\bm v\) by \(\bm u\), we may therefore assume
from now on that
\begin{equation}\label{eq:B-v-nonnegative}
    v_i>0
    \quad\text{and}\quad
    B_{ij}\ge0
    \qquad\text{for all }i,j.
\end{equation}
To obtain strict positivity, we show that no entry of $\bm B$ can be zero. Suppose $B_{kl}=0$ for some $(k,l)$. Define a new matrix $\bm B'$ by
\[
\bm B'= \sqrt{1-\delta^2} \cdot \bm B + \delta \cdot \bm 1_{kl}
\]
where $\bm 1_{kl}$ is a symmetric matrix with the $(k,l)$ entry equal to one and all other entries equal zero.

One can check that
\[
\mc I_{U} \lb \bm B' \rb =  
\begin{cases}
    (1-\delta^2) \cdot \mc I_{U} \lb \bm B \rb + \delta^2 \cdot \left[ \mathbf 1_{\la k \neq l \ra} + \frac{\mathbf{1}_{\la k=l \ra}}{a} \right], \qquad &\text{if $\la k,l \ra \subset U$}; \\
    (1-\delta^2) \cdot \mc I_{U} \lb \bm B \rb, &\text{if $\la k,l \ra \not\subset U$}
\end{cases}
\]
In either case, $\mc I_{U} \lb \bm B' \rb \leq |U|$ for every $U \subset [m]$ whenever $\delta$ is sufficiently small. Thus, $\bm B' \in \mc K_m(a)$. However, this leads to a contradiction because
\[
\gamma_m(a) \geq \lmax \lb \bm B'  \rb \geq \bm v^\top \, \bm B' \, \bm v = \sqrt{1-\delta^2} \cdot \gamma_m(a) +2 \delta\cdot v_k \, v_l > \gamma_m(a)
\]
for all sufficiently small $\delta$, since $v_kv_l>0$.

\medskip
\noindent
\underline{Step 2: Hierarchical structure of tight sets.}
Call a nonempty set \(U\subset[m]\) {\it tight} if $\mathcal I_U(\bm B)=|U|$.
We use the convention \(\mathcal I_\varnothing(\bm B)=0\). For arbitrary
\(U,V\subset[m]\), we have 
\begin{align}
&\mathcal I_{U\cup V}(\bm B)
+
\mathcal I_{U\cap V}(\bm B)
\nonumber
=
\mathcal I_U(\bm B)
+
\mathcal I_V(\bm B)
+
\sum_{\substack{i\in U\setminus V\\
                  j\in V\setminus U}}
B_{ij}^2.
\label{eq:uncrossing-identity}
\end{align}
If \(U\) and \(V\) are tight and neither contains the other, then both
\(U\setminus V\) and \(V\setminus U\) are nonempty. Hence
\[
\begin{aligned}
    \mathcal I_{U\cup V}(\bm B) + \mathcal I_{U\cap V}(\bm B)
    &> |U|+|V|= |U\cup V|+|U\cap V|,
\end{aligned}
\]
which is a contradiction. Thus, any two nonempty tight sets are nested.

We next claim that every vertex belongs to a tight set. Suppose that
some \(i\in[m]\) belongs to no tight set. Then for all $U \subset [m] \setminus \la i \ra$, we have strict inequalities
\[
\mc I_{U \cup \la i \ra} \lb \bm B \rb < 1+|U|.
\]
Since there are finitely many such $U$, we may construct a new matrix $\bm B'' \in \mc K_{m}(a)$ by increasing $B_{ii}$ slightly. The resulting matrix $\bm B''$ satisfies
\[
\lmax \lb \bm B'' \rb > \lmax \lb \bm B \rb = \gamma_m(a),
\]
which is a contradiction. Thus, every vertex belongs to a tight set.
The tight sets therefore form a chain whose union is \([m]\). Its
largest member must consequently be \([m]\), so $ \mathcal I_{[m]}(\bm B)=m.$ $\hfill$ $\square$

\begin{lemma}\label{lem:chain}
Assume that \(a>2\) and \(m\ge2\). If
\(\bm B\in\mathcal K_m(a)\) attains \(\gamma_m(a)\), then there exists
\(U\subset[m]\) with \(|U|=m-1\) such that
\[
    \mathcal I_U(\bm B)=m-1.
\]
\end{lemma}

\noindent\textbf{Proof of Lemma \ref{lem:chain}.}
By Steps 2 and 3 in the proof of Lemma~\ref{lem:structure}, after a
diagonal sign conjugation we may assume that
\[
    B_{ij}>0
    \qquad\text{for every }i,j\in[m],
\]
and the nonempty tight sets are totally ordered by inclusion.

Let \(U\) be the largest proper tight set, with the convention
\(U=\varnothing\) if no nonempty proper tight set exists, and put
\[
    W:=[m]\setminus U.
\]
By Lemma~\ref{lem:structure}, \([m]\) is tight. Moreover,
\begin{equation}\label{eq:proper-W-constraints-strict}
    \mathcal I_V(\bm B)<|V|
    \qquad
    \text{for every }V\subsetneq[m]
    \text{ such that }V\cap W\ne\varnothing.
\end{equation}
Indeed, if such a set \(V\) were tight, then \(V\not\subseteq U\);
the chain property would therefore imply \(U\subsetneq V\), contradicting
the maximality of \(U\).

Suppose that \(|W|\ge2\), and choose distinct \(i,j\in W\). Let
\(\bm Q_\theta\) be the identity outside the coordinates \(i,j\), and
let its \((i,j)\)-block be
\[
    \begin{pmatrix}
        \cos\theta&-\sin\theta\\
        \sin\theta&\cos\theta
    \end{pmatrix}.
\]
Set $\bm B_\theta:=\bm Q_\theta^\top\bm B\bm Q_\theta$. Its \((i,j)\)-entry is
\[
    (B_\theta)_{ij}
    =
    B_{ij} \cdot \cos(2\theta)
    +
    \frac{B_{jj}-B_{ii}}{2} \cdot \sin(2\theta).
\]
Since \(B_{ij}>0\), there exist arbitrarily small
\(\theta\ne0\) such that
\begin{equation}\label{eq:rotation-reduces-offdiag}
    |(B_\theta)_{ij}|<B_{ij}.
\end{equation}
Orthogonal conjugation preserves the Frobenius norm. Furthermore, the
Frobenius norm of the \((i,j)\)-principal block is preserved, so
\[
    (B_\theta)_{ii}^2+(B_\theta)_{jj}^2
    -
    B_{ii}^2-B_{jj}^2
    =
    2\left[
        B_{ij}^2-(B_\theta)_{ij}^2
    \right].
\]
Thus
\[
\begin{aligned}
    \mathcal I_{[m]}(\bm B_\theta)
    -
    \mathcal I_{[m]}(\bm B)
    &=
    \left(\frac2a-1\right)
    \left[
        B_{ij}^2-(B_\theta)_{ij}^2
    \right]
    <0,
\end{aligned}
\]
where we used \(a>2\) and
\eqref{eq:rotation-reduces-offdiag}. Hence,
$\mathcal I_{[m]}(\bm B_\theta)<m$ for all sufficiently small $\theta$.

The rotation acts only on coordinates in \(W\), so every constraint
indexed by a subset of \(U\) is unchanged for all $\theta$. All other proper constraints
are strict at \(\bm B\) by
\eqref{eq:proper-W-constraints-strict}. Moreover, since
\(\bm B_\theta\to\bm B\) as \(\theta\to0\), they remain strict for
all sufficiently small \(\theta\). Therefore, for all sufficiently small $\theta$,
$\bm B_\theta\in\mathcal K_m(a)$.

On the other hand, \(\bm B_\theta\) is orthogonally similar to
\(\bm B\), and hence
\[
    \lambda_{\max}(\bm B_\theta)
    =
    \lambda_{\max}(\bm B)
    =
    \gamma_m(a).
\]
Thus, \(\bm B_\theta\) is also an optimizer. Lemma~\ref{lem:structure}
would then imply $\mathcal I_{[m]}(\bm B_\theta)=m$,
contradicting the strict inequality above. Thus, \(|W|=1\) and $\mathcal I_U(\bm B)=m-1$.
\(\hfill\square\)

\begin{prop}\label{prop:gamma-recursion}
Assume that \(a>2\). Then $\gamma_1(a)=\sqrt a$,
and, for every \(m\ge2\),
\begin{equation}\label{eq:gamma-recursion-proof}
 \gamma_m(a)
 =
 \max_{0\le t\le1}
 \left\{
 t\gamma_{m-1}(a)
 +
 \sqrt{(1-t)\bigl(a(1-t)+4t\bigr)}
 \,
 \right\}.
\end{equation}
\end{prop}

\noindent \textbf{Proof of Proposition \ref{prop:gamma-recursion}.}
The identity \(\gamma_1(a)=\sqrt a\) follows immediately from the formula
\[
    \mathcal I_{\{1\}}(B)=\frac{B_{11}^2}{a}\le 1.
\]
Let \(m\ge2\), and let \( \bm B\in\mathcal K_m(a)\) attain
\(\gamma_m(a)\). By Lemma~\ref{lem:structure}, after relabeling we may
write
\[
    \bm B=
    \begin{pmatrix}
        \bm C & \bm g\\
        \bm g^\top & d
    \end{pmatrix},
\]
where
\[
    \bm C\in\mathcal K_{m-1}(a),
    \qquad
    \mathcal I_{[m-1]}(\bm C)=m-1,
    \qquad
    \mathcal I_{[m]}(\bm B)=m.
\]
Subtracting the last two identities gives
\begin{equation}\label{eq:last-layer-budget}
    \|\bm g\|_2^2+\frac{d^2}{a}=1.
\end{equation}
We will show that
\begin{align} \label{<=}
    \gamma_m(a) \leq \max_{0\le t\le1} \left\{  t\gamma_{m-1}(a) + \sqrt{(1-t)\bigl(a(1-t)+4t\bigr)} \, \right\}.
\end{align}
and
\begin{align} \label{>=} 
\gamma_m(a) \geq \max_{0\le t\le1} \left\{  t\gamma_{m-1}(a) + \sqrt{(1-t)\bigl(a(1-t)+4t\bigr)} \, \right\}.
\end{align}

\noindent \underline{\it Proof of \eqref{<=}.} Let \((\bm u,s)\in\mathbb R^{m-1}\times\mathbb R\) be a unit top
eigenvector of \(\bm B\), and put
\[
    t:=\|\bm u\|_2^2,
    \qquad
    s^2=1-t.
\]
Since \( \bm C\in\mathcal K_{m-1}(a)\),
\[
    \bm u^\top \bm C \bm u = \| \bm u \|^2 \cdot \frac{\bm u^\top}{\| \bm u \|} \, \bm C \, \frac{\bm u}{\| \bm u \|}
    \le
    t \cdot \gamma_{m-1}(a).
\]
Moreover, by \eqref{eq:last-layer-budget} and Cauchy--Schwarz,
\begin{align*}
    2s\cdot \bm g^\top \bm u+d s^2
    =
    \left\langle
        \left(\bm g,\frac d{\sqrt a}\right),
        \left(2s\bm u,\sqrt a\,s^2\right)
    \right\rangle
    &\le
    \sqrt{\| \bm g\|_2^2+\frac{d^2}{a}}
    \cdot 
    \sqrt{4s^2\|\bm u\|_2^2+a s^4} \\
    &=
    \sqrt{(1-t)\bigl(4t+a(1-t)\bigr)}.
\end{align*}
Therefore,
\[
\begin{aligned}
    \gamma_m(a)
    =
    \bm u^\top \, \bm C \, \bm u+2s\cdot \bm g^\top \bm u+d s^2
    &\le
    t\gamma_{m-1}(a)
    +
    \sqrt{(1-t)\bigl(4t+a(1-t)\bigr)} \\
    &\leq \max_{0\le t\le1}
 \left\{
 t\gamma_{m-1}(a)
 +
 \sqrt{(1-t)\bigl(a(1-t)+4t\bigr)}
 \,
 \right\}.
\end{aligned}
\]
\noindent \underline{\it Proof of \eqref{>=}.} Let \( \bm C\in\mathcal K_{m-1}(a)\) attain \(\gamma_{m-1}(a)\), and let \(\bm q\) be a unit top eigenvector of \(\bm C\). For \(0\le t<1\), define
\[
    R(t)
    :=
    \sqrt{(1-t)\bigl(4t+a(1-t)\bigr)},
\]
\[
    \bm g_t
    :=
    \frac{2\sqrt{t(1-t)}}{R(t)} \bm q,
    \qquad
    d_t
    :=
    \frac{a(1-t)}{R(t)}.
\]
Note that
\[
\begin{aligned}
    \|\bm g_t\|_2^2+\frac{d_t^2}{a} &= \frac{4t(1-t)+a(1-t)^2}{R(t)^2}=1.
\end{aligned}
\]
Define
\[
    \bm B_t
    :=
    \begin{pmatrix}
        \bm C & \bm g_t\\
        \bm g_t^\top & d_t
    \end{pmatrix}.
\]
We will show that \(\bm B_t\in\mathcal K_m(a)\). Constraints not involving the
last coordinate follow from \(\bm C\in\mathcal K_{m-1}(a)\). For constraints involving the last coordinate, let
\(T\subset[m-1]\). Then
\[
\begin{aligned}
    \mathcal I_{T\cup\{m\}}(\bm B_t)
    &=
    \mathcal I_T(\bm C)
    +
    \frac{d_t^2}{a}
    +
    \sum_{i\in T}g_{t,i}^2
    \le
    |T|
    +
    \frac{d_t^2}{a}
    +
    \|\bm g_t\|_2^2
    =|T|+1.
\end{aligned}
\]
Thus, \(\bm B_t\in\mathcal K_m(a)\).

Evaluated at the unit vector $\lb \sqrt{t} \cdot \bm q, \sqrt{1-t} \rb^\top$,
the Rayleigh quotient of $\bm B_t$ is
\[
\begin{aligned}
t\gamma_{m-1}(a)
+
2\sqrt{t(1-t)}\, \bm g_t^\top \bm q
+
d_t(1-t) &=
t\gamma_{m-1}(a)
+
\frac{
4t(1-t)+a(1-t)^2
}{
R(t)
}\\
&=
t\gamma_{m-1}(a)+R(t).
\end{aligned}
\]
Consequently,
\[
    \gamma_m(a)
    \ge
    \lb \sqrt{t} \cdot \bm q, \sqrt{1-t} \rb^\top \, \bm B_t \, \lb \sqrt{t} \cdot \bm q, \sqrt{1-t} \rb
    \geq 
    t\gamma_{m-1}(a)
    +
    \sqrt{(1-t)\bigl(4t+a(1-t)\bigr)}.
\]
for all $t \in [0,1)$. Thus, we obtain \eqref{>=}. This completes the proof. $\hfill$ $\square$

We next derive a geometric description of the maximizer in $\mc K_m(a)$. It implies that when $a>2$, the maximizing matrices in different dimensions are nested in an appropriate sense.

\begin{prop}\label{prop:complete-tight-chain}
Assume that \(a>2\). If \( \bm B\in\mathcal K_m(a)\) attains
\(\gamma_m(a)\), then there exists a permutation \(\pi\) of \([m]\)
such that, with
\[
    U_r:=\{\pi(1),\ldots,\pi(r)\},
    \qquad r=1,\ldots,m,
\]
we have
\[
    \mathcal I_{U_r}(\bm B)=r,
    \qquad \forall  r=1,\ldots,m.
\]
\end{prop}

\noindent \textbf{Proof of Proposition \ref{prop:complete-tight-chain}.}
We argue by induction on \(m\). The assertion is immediate for \(m=1\). 
Suppose \(m\ge2\) and that the result holds for \(m-1\). By
Lemma~\ref{lem:structure}, after relabeling we may write
\[
    \bm B=
    \begin{pmatrix}
        \bm C & \bm g\\
        \bm g^\top & d
    \end{pmatrix},
\]
where
\[
    \bm C\in\mathcal K_{m-1}(a),
    \qquad
    \mathcal I_{[m-1]}(\bm C)=m-1,
    \qquad
    \mathcal I_{[m]}(\bm B)=m.
\]
We will show that $\bm C$ is also a maximizer for the largest eigenvalue over $\mc K_{m-1}(a)$. As in the proof of Proposition \ref{prop:gamma-recursion}, let \((\bm u,s)\) be a unit top eigenvector of \(\bm B\), and put
\[
    t:=\|\bm u\|_2^2,
    \qquad
    s^2=1-t.
\]
Subtracting the last two identities gives
\begin{equation*}
    \|\bm g\|_2^2+\frac{d^2}{a}=1.
\end{equation*}
We first note that $0<t<1$. Indeed, if \(t=1\), then
\[
    \gamma_m(a)
    =
    \bm u^\top \bm C \bm u
    \le
    \gamma_{m-1}(a),
\]
contradicting Lemma~\ref{lem:gamma-strict}. If \(t=0\), then
\[
    \gamma_m(a)=d\le\sqrt a=\gamma_1(a),
\]
which gives the same contradiction.

As in the proof of Proposition~\ref{prop:gamma-recursion},
\[
\begin{aligned}
    \gamma_m(a)
    =
    \bm u^\top \bm C \bm u+2s\cdot \bm g^\top \bm u+d s^2
    &\le
    \max_{0\le r\le1}
    \left\{
        r \cdot \gamma_{m-1}(a)
        +
        \sqrt{(1-r)\bigl(4r+a(1-r)\bigr)}
    \right\}
    = \gamma_m(a),
\end{aligned}
\]
where the final equality follows from
\eqref{eq:gamma-recursion-proof}. Hence, equality holds throughout.

The first inequality was obtained by adding the two bounds
\[
    \bm u^\top \bm C \bm u\le t\gamma_{m-1}(a)
\]
and
\[
    2s\cdot \bm g^\top \bm u+d s^2
    \le
    \sqrt{(1-t)\bigl(4t+a(1-t)\bigr)}.
\]
Since equality holds, we must have $\bm u^\top \bm C \bm u=t\gamma_{m-1}(a)$, or equivalently,
\[
    \left(\frac{\bm u}{\sqrt t}\right)^\top
    \bm C
    \left(\frac{\bm u}{\sqrt t}\right)
    =
    \gamma_{m-1}(a).
\]
Since \(\bm C\in\mathcal K_{m-1}(a)\), this yields $\lmax(\bm C)=\gamma_{m-1}(a)$.
Thus, $\bm C$ is also a maximizer for the largest eigenvalue over $\mc K_{m-1}(a)$.

Now, by the induction hypothesis, there exists a permutation \(\pi_0\) of
\([m-1]\) such that
\[
    \mathcal I_{\{\pi_0(1),\ldots,\pi_0(r)\}}(C)=r,
    \qquad r=1,\ldots,m-1.
\]
The same identities hold for the corresponding principal submatrices
of \(B\). Define a new permutation \(\pi_1\) of \([m]\) by
\[
    \pi_1(r):=\pi_0(r),
    \qquad r=1,\ldots,m-1,
    \qquad
    \pi_1(m):=m.
\]
Then
\[
    \mathcal I_{\{\pi(1),\ldots,\pi(r)\}}(B)=r,
    \qquad r=1,\ldots,m.
\]
This completes the proof. $\hfill$ $\square$

\section{Proof of Theorem \ref{thm:c=0}}

Write $$L_{n,k}:= \log \lb n/k \rb$$ for shorthand.

Without loss of generality, we may assume that $a=2$. In fact, if $\bm A$ is the random matrix with diagonal variance $a$ and $\bm A_{\rm GOE}$ is the matrix obtained by replacing the diagonal entries of $\bm A$ by $\sqrt{2} \cdot Z_i$, where $\la Z_1,\dots, Z_n \ra$ are standard normal random variables, then
\[
\left| \frac{M_{n,k}}{2\sqrt{k\cdot L_{n,k}}} - \frac{M_{n,k}^{\rm GOE}}{2\sqrt{k \cdot L_{n,k}}}  \right | \leq  \frac{\max_{1\leq i \leq n} \left| A_{ii} - \sqrt{2} \cdot Z_i \right|}{2\sqrt{k \cdot L_{n,k}}} = O_{\mb P} \lb \sqrt{\frac{\log n}{k\cdot \log \lb n/k \rb}} \, \rb = o_{\mb P}(1)
\]
where the second equality follows from the fact that the maximum of n i.i.d. standard normal random variables is of order $O_{\mb P} \lb \sqrt{\log n} \rb$. Thus, we may assume $a=2$ for simplicity. This assumption simplifies the covariance computation below.

It suffices to prove that for every $1>\delta>0$,
\begin{align} \label{c=0 upper bound}
\mb P \lb  M_{n,k} \geq (1+\delta)\cdot 2\sqrt{k \cdot L_{n,k}} \, \rb \to 0    
\end{align}
and 
\begin{align} \label{c=0 lower bound}
\mb P \lb  M_{n,k} \leq (1-\delta)\cdot 2\sqrt{k \cdot L_{n,k}} \, \rb \to 0    
\end{align}

\noindent \underline{Proof of \eqref{c=0 upper bound}.} The upper bound \eqref{c=0 upper bound} is standard and follows from an $\ve$-net argument. For a subset $S \subset [n]$ and a sufficiently small $\ve$, to be chosen below, let $\mc N_S$ be an $\ve$-net of the unit sphere $\mb S^{|S|-1}$. It is well known that
\[
\left| \mc N_\ve \right| \leq \lb 1 + 2/\ve \rb^s. 
\]
Observe that for every symmetric matrix $\bm B \in \mb R^{s \times s}$, we have
\[
\| \bm B \|_{\rm op} \leq \frac{1}{1-2\ve} \cdot \max_{\bm x \in \mc N_S} \left| \bm x^\top \bm B \bm x  \right|.
\]
Therefore, for a constant $C_\ve<1+\delta$ to be chosen below as a function of $\ve, \delta$, by applying the above bound to $\la \bm A_S; |S| \leq k \ra$, 
\begin{align*}
    \mb P \lb M_{n,k} > 2\lb 1 + \delta \rb \cdot \sqrt{kL_{n,k}} \rb &\leq \sum_{s=1}^k \sum_{|S|=s} \sum_{\bm x \in N_S}
    \mb P \lb  \left| \bm x^\top \bm A \bm x  \right|> 2C_\ve \cdot \sqrt{kL_{n,k}} \, \rb \\
    &\leq 2 \cdot \exp \lb -C_\ve^2 \cdot kL_{n,k}  \rb \cdot \sum_{s=1}^k \binom{n}{s} \cdot \lb 1 + \frac 2\ve \rb^s 
\end{align*}
where the second inequality follows from the fact that, for every $\bm x$, $\bm x^\top \bm A \bm x \sim N(0,2)$, together with the tail bound
\[
\mb P \lb |Z|>x \rb \leq e^{-x^2/4}, \qquad Z \sim N(0,2). 
\]
Since $k=o(n)$, we can use a crude upper bound $\binom{n}{s} \leq \lb en/s \rb^s$ to get
\[
 \sum_{s=1}^k \binom{n}{s} \cdot \lb 1 + \frac 2\ve \rb^s \leq \sum_{s=1}^k \left[ \frac{en}{s} \cdot  \lb 1 + \frac 2\ve \rb \right]^s 
 = \sum_{s=1}^k \exp \left[ s \cdot \log \lb \frac{en}{s} \rb + s \cdot \log \lb 1 + \frac 2\ve \rb \right]
\]
Since the function $s \to s\cdot \log(en/s)$ is increasing for sufficiently large $n$, we deduce that
\[
 \sum_{s=1}^k \binom{n}{s} \cdot \lb 1 + \frac 2\ve \rb^s \leq  \exp \left[ k \cdot \log \lb \frac{en}{k} \rb + k \cdot \log \lb 1 + \frac 2\ve \rb + \log k \right]
\]
Thus
\[
    \mb P \lb M_{n,k} > 2\lb 1 + \delta \rb \cdot \sqrt{kL_{n,k}} \rb \leq 2 \cdot \exp \la k\cdot \left[ -C_\ve^2\cdot L_{n,k} + 1 + L_{n,k} + \log \lb 1 + \frac 2\ve \rb \right] + \log k \ra.
\]
Since $L_{n,k} \to \infty$, the last term tends to zero provided that $C_\ve>1 $. Thus, we can choose $C_\ve=(1-2\ve)(1+\delta)$ and $\ve$ small enough such that $C_\ve>1$.

\noindent \underline{Proof of \eqref{c=0 lower bound}.} To show the lower bound, consider the subset of the sparse sphere
\[
\mc V_{n,k}:= \la \bm x \in \mb R^n: \left| \mbox{supp}(\bm x) \right| = k, \, x_i \in \la -\frac{1}{\sqrt{k}}, \frac{1}{\sqrt{k}} \ra \, \text{on} \, \mbox{supp}(\bm x) \ra.
\]
It is easy to check that
\[
|\mc V_{n,k}| = 2^k \cdot \binom{n}{k}.
\]
Put
\[
O_{n,k}:= \log |\mc V_{n,k}| = kL_{n,k} + O(k).
\]
Let $\la Y_{\bm x}; \bm x \in \mc V_{n,k} \ra$ be the Gaussian process defined by
\[
Y_{\bm x}:= \frac{1}{\sqrt{2}}\cdot \bm x^\top \bm A \bm x.
\]
Let $Z_n$ be the counting random variable
\[
Z_n:= \sum_{\bm x \in \mc V_{n,k}} \mathbf{1}_{\la Y_{\bm x} > z_n \ra}, \qquad z_n:=(1-\delta)\sqrt{2O_{n,k}}
\]
To prove \eqref{c=0 lower bound}, it suffices to show that $\mb P \lb Z_n > 0 \rb \to 1$. Thus, by the Paley--Zygmund inequality, as in the proof of \eqref{lower}, we only need to prove
\[
\frac{\mb E Z_n^2}{\lb \mb E Z_n \rb^2} \to 1. 
\]
We next derive a useful representation for the ratio in the last display. Observe that
\begin{align*}
    \frac{\mb E Z_n^2}{\lb \mb E Z_n \rb^2} &= \frac{\sum_{\bm x, \bm y \in \mc V_{n,k}}   \mb P \lb Y_{\bm x} > z_n, Y_{\bm y} > z_n \rb}{|\mc V_{n,k}|^2 \cdot \overline{\Phi}(z_n)^2}.
\end{align*}
For a pair $(\bm x, \bm y)$, the covariance structure of $Y$ is $\mb E \lb Y_{\bm x} Y_{\bm y} \rb= \langle \bm x, \bm y \rangle^2$; hence,
\[
\begin{pmatrix}
    Y_{\bm x} \\ Y_{\bm y} 
\end{pmatrix}
\sim N \lb \bm 0, 
\begin{pmatrix}
    1 & \langle \bm x, \bm y \rangle^2 \\
    \langle \bm x, \bm y \rangle^2 & 1
\end{pmatrix}
\rb.
\]
Therefore, with $\bm X , \bm Y$ i.i.d. and uniformly distributed on $\mc V_{n,k}$, we can write
\[
 \frac{\mb E Z_n^2}{\lb \mb E Z_n \rb^2} = \mb E \mc R_{z_n} \left[  \langle \bm X, \bm Y \rangle^2 \right]  .
\]
Here
\begin{align} \label{mc R}
\mc R_{z}(\rho) = \frac{\mb P \lb  G_1>z_n, G_2>z_n \rb}{\overline{\Phi} \lb z_n \rb^2}, \qquad 
\begin{pmatrix}
    G_1 \\ G_2 
\end{pmatrix}
\sim N \lb \bm 0, \,
\begin{pmatrix}
    1 & \rho \\
    \rho & 1
\end{pmatrix}
\rb.
\end{align}
Let $S_n$ be the sum of $n$ i.i.d. Rademacher random variables. We can then represent the random inner product as
\[
 \langle \bm X, \bm Y \rangle \stackrel{d}{=} \frac{S_R}{k}, \qquad R\sim \mbox{Hypergeometric} \lb N=n,K=k, \text{draws}=k \rb
\]
and $R$ is independent of the Rademacher random variables in $S_n$. Thus,
\begin{align*}
     \frac{\mb E Z_n^2}{\lb \mb E Z_n \rb^2} = \mb E \mc R_{z_n} \lb \frac{S_R^2}{k^2} \rb.
\end{align*}
To bound the last term, we split the expectation into the small- and large-overlap regions. To this end, put
\[
r_n^*:= \lfloor c_0k/L_{n,k}  \rfloor \qquad \text{for some sufficiently small $c_0>0$}.
\]
Write 
\[
 \mb E \mc R_{z_n} \lb \frac{S_R^2}{k^2} \rb -1  \leq    \underbrace{  \mb E \la  \left[ \mc R_{z_n} \lb \frac{S_R^2}{k^2} \rb  -1 \right] \cdot \mathbf{1}_{\la R \leq r_n^* \ra} \ra }_{I}
 +
 \underbrace{ \mb E \left[ \mc R_{z_n} \lb \frac{S_R^2}{k^2} \rb \cdot \mathbf{1}_{\la R > r_n^* \ra}  \right]
  }_{II} .
\]

\underline{ \it Bounding the small-overlap expectation I.} Put $\lambda_n:= z_n^2/k^2$ and note that $\lambda_nr_n^* \leq 1/16$ for all sufficiently large $n$ if $c_0$ is chosen sufficiently small. By Lemma \ref{lem: mc R bound}, we have
\begin{align*}
    \mb E \left[ \mc R_{z_n} \lb \frac{S_r^2}{k^2} \rb  -1 \Big| R=r \right] \cdot \mathbf{1}_{\la r \leq r_n^* \ra}
    &\lesssim  \mb E \left[ z_n^2 \cdot \frac{S_r^2}{k^2} \cdot \exp \lb \frac{z_n^2S_r^2}{k^2} \rb  \Big| R=r \right]  \cdot \mathbf{1}_{\la r \leq r_n^* \ra} \\
    &= \lambda_n \cdot  \mb E \left[  S_r^2  \cdot \exp \lb \lambda_n  S_r^2 \rb \Big| R=r  \right] \cdot \mathbf{1}_{\la r \leq r_n^* \ra} \\
    &\lesssim \lambda_n \cdot \mb E \left[ 8r \cdot \lb e^{S_r^2/8r} -1 \rb \cdot \exp \lb \lambda_n  S_r^2 \rb \Big| R=r \right] \cdot \mathbf{1}_{\la r \leq r_n^* \ra} \\
    &\lesssim \lambda_nr  \cdot \mb E \left[ \exp \lb \lb \lambda_n+ \frac{1}{8r} \rb \cdot S_r^2 \rb  \right] \cdot \mathbf{1}_{\la r \leq r_n^* \ra}.
\end{align*}
To bound the last term, take $u>0$; let $G$ be a standard normal random variable independent of everything else, and write
\begin{align*}
    \mb E \left[ \exp \lb u S_r^2  \rb \right]& = \mb E \left[ \mb E \lb \exp\lb  \sqrt{2u} \cdot GS_r \rb \Big| S_r \rb \right] \\
    &= \mb E \left[ \mb E \lb \exp\lb  \sqrt{2u} \cdot GS_r \rb \Big| G \rb \right] \\
    &= \mb E \left[ \cosh \lb \sqrt{2u} \cdot G \rb^r \right]  \leq  \mb E \exp \lb urG^2 \rb
    \leq \lb 1-2ur \rb^{-1/2}. 
\end{align*}
Applying the above estimate with $u=\lambda_n + 1/(8r)$, we obtain 
\[
 \mb E \left[ \exp \lb \lb \lambda_n+ \frac{1}{8r} \rb \cdot S_r^2 \rb  \right] \leq \frac{1}{\sqrt{1-r(2\lambda_n+1/(4r))}} \leq 10
\]
since $\lambda_nr \leq \lambda_n r_n^*\leq 1/16$.

Consequently,

\begin{align*}
     \mb E \la \left[ \mc R_{z_n} \lb \frac{S_R^2}{k^2} \rb -1  \right] \cdot \mathbf{1}_{\la R \leq r_n^* \ra} 
     \ra \lesssim \lambda_n \cdot  \mb E R = & \frac{z_n^2}{k^2} \cdot \frac{k^2}{n}
     = \frac{z_n^2}{n} \lesssim \frac{k\cdot \log(n/k)}{n} \to 0. 
\end{align*}

\underline{ \it Bounding the large-overlap expectation II.} By Lemma \ref{lem: mc R bound}, for all sufficiently large $n$,
\begin{align*}
    \mb E \left[ \mc R_{z_n} \lb \frac{S_R^2}{k^2} \rb \Big| R=r \right] & \leq 
     \mb E \left[ \mc R_{z_n} \lb \frac{R^2}{k^2} \rb \Big| R=r \right] \\
    &\lesssim z_n \cdot \exp \la \frac{z_n^2 \cdot \lb r^2/k^2 \rb}{1+ \lb r^2/k^2 \rb} \ra
    = z_n \cdot \exp \la \frac{z_n^2 \cdot r^2 }{k^2+ r^2 } \ra
\end{align*}
where the first line follows from the first bullet point in Lemma \ref{lem: mc R bound} and the second line follows from the third bullet point in Lemma \ref{lem: mc R bound}. 

Consequently,
\begin{align*}
 \mb E \left[ \mc R_{z_n} \lb \frac{S_R^2}{k^2} \rb \cdot \mathbf{1}_{\la R > r_n^* \ra}  \right]
 &= \sum_{r>r_n^*} \mb E \left[ \mc R_{z_n} \lb \frac{S_R^2}{k^2} \rb \Big| R=r \right] \cdot \mb P \lb R=r \rb \\
 &\lesssim \sum_{r>r_n^*} \mb P \lb R=r \rb \cdot z_n \cdot \exp \la \frac{z_n^2 \cdot r^2 }{k^2+ r^2 } \ra \\
 &\lesssim \sum_{r>r_n^*} \lb \frac{e^2k^2}{nr} \rb^r \cdot z_n \cdot \exp \la \frac{z_n^2 \cdot r^2 }{k^2+ r^2 } \ra \\
 &= \sum_{r>r_n^*} \exp \la r\cdot \log \lb \frac{e^2k^2}{nr}  \rb + \log z_n + \frac{z_n^2 \cdot r^2 }{k^2+ r^2 }  \ra
\end{align*}
To bound the last term, put $\theta:=r/k$ and write 
\begin{align*}
    & \exp \la r\cdot \log \lb \frac{e^2k^2}{nr}  \rb + \log z_n + \frac{z_n^2 \cdot r^2 }{k^2+ r^2 }  \ra \\
    = & \exp \la  2r -r\cdot L_{n,k} - r\log \theta + \log z_n +  \frac{\theta^2}{1+\theta^2}\cdot 2(1-\delta)^2\cdot k L_{n,k}\cdot (1+o(1)) \ra \\
    =& \exp \la  2r -r\cdot L_{n,k} - r\log \theta + \log z_n +  \frac{\theta}{1+\theta^2}\cdot 2(1-\delta)^2\cdot r L_{n,k}\cdot (1+o(1)) \ra \\
    =& \exp \la -r \cdot \left[ -2 + L_{n,k} +\log \theta - \frac{2\theta(1-\delta)^2}{1+\theta^2} \cdot L_{n,k} \cdot (1+o(1))  \right] + \log z_n \ra
\end{align*}
where the second line follows from $z_n^2=2(1-\delta)^2 \cdot kL_{n,k} \cdot (1+o(1))$ and the third line follows from substituting $\theta=r/k$.

Since $2\theta/(1+\theta^2) \leq 1$, $\theta \geq r_n^*/k > c_0/(L_{n,k})$, and $L_{n,k} \to \infty$, there exists a constant $c_\delta>0$ depending only on $\delta$ such that, for all sufficiently large $n$,
\[
-2 + L_{n,k} +\log \theta - \frac{2\theta(1-\delta)^2}{1+\theta^2} \cdot L_{n,k} \cdot (1+o(1)) \geq c_\delta \cdot L_{n,k}.
\]
Thus
\begin{align*}
\sum_{r>r_n^*} \exp \la r\cdot \log \lb \frac{e^2k^2}{nr}  \rb + \log z_n + \frac{z_n^2 \cdot r^2 }{k^2+ r^2 }  \ra &\leq z_n \cdot \sum_{r \geq r_n^*+1} \exp \la -rc_\delta L_{n,k}  \ra \\
&=z_n \cdot \frac{ \exp \la -c_\delta r_n^* L_{n,k} \ra}{1- \exp \la -c_\delta L_{n,k}\ra} \\
&= O \lb \sqrt{k L_{n,k}} \cdot \exp \la -c_\delta r_n^* L_{n,k} \ra \rb
\end{align*}
To bound the last term, note that
\[
\sqrt{k L_{n,k}} \cdot \exp \lb -c_\delta r_n^* L_{n,k} \rb \leq \max \la \sqrt{\frac{2}{c_0}} \cdot L_{n,k} \cdot \exp \lb -c_\delta L_{n,k} \rb, \sqrt{\frac{c_0}{2}} \cdot k \cdot \exp \lb -c_\delta c_0k \rb \ra .
\]
Indeed, if $r_n^* \leq 2 $, then $k \leq (2/c_0)L_{n,k}$; otherwise, $L_{n,k} \leq (c_0/2)k$. The last display tends to zero since $\min \la k, L_{n,k} \ra \to \infty$. We therefore obtain \eqref{c=0 lower bound}, completing the proof. $\hfill$ $\square$



\section{Proof of Theorem \ref{c>0}}
We use the following version of Fekete's lemma.

\begin{lemma} \label{lem:Fekete}
    If a sequence of real numbers $\la a_n; n \geq 1 \ra$ satisfies
    \[
    a_{m+n} \geq a_m + a_n
    \]
    for all $m,n \geq 1,$
    then 
    \[
    \lim_{n \to \infty} \frac{a_n}{n} = \sup_{n \geq 1} \la  \frac{a_n}{n} \ra.
    \]
\end{lemma}
Lemma \ref{lem:Fekete} states that the limit $a_n/n$ exists and equals the supremum of the set $\la a_n/n; n \geq 1\ra$. We now turn to the proof of Theorem \ref{c>0}. Without loss of generality, we may assume $a=2$. In fact, if $\bm A$ is the random matrix with diagonal variance $a$ and $\bm A_{\rm GOE}$ is the matrix obtained by replacing the diagonal entries of $\bm A$ by $\sqrt{2} \cdot Z_i$, where $\la Z_1,\dots, Z_n \ra$ are standard normal random variables, then
\[
\left| \frac{M_{n,k}}{\sqrt{n}} - \frac{M_{n,k}^{\rm GOE}}{\sqrt{n}}  \right | = \frac{\max_{1\leq i \leq n} \left| A_{ii} - \sqrt{2} \cdot Z_i \right|}{\sqrt{n}} = O_{\mb P} \lb \sqrt{\frac{\log n}{n}} \, \rb = o_{\mb P}(1)
\]
where the second equality follows from the fact that the maximum of n i.i.d. standard normal random variables is of order $O_{\mb P} \lb \sqrt{\log n} \rb$. Thus, we may assume $a=2$ for simplicity. This assumption simplifies the covariance computation below.

Define
\begin{align*}
\Sigma_{n,k}:&= \la \bm \sigma \in \mb R^{n}: \, \| \bm \sigma \|_2=\sqrt{n} \,\, \text{and} \,\,  \| \bm \sigma \|_0 \leq k  \ra,  \\
H_{n,k} \lb \bm \sigma \rb :&= n^{-1/2} \cdot \bm \sigma^\top \, \bm A \, \bm \sigma, \\
\mc H_{n,k}:&= \mb E \lb \max_{\bm \sigma \in \Sigma_{n,k}} H_{n,k} \lb \bm \sigma \rb \rb.
\end{align*}
It is easy to see that
\begin{align} \label{equality}
\frac{\mb E M_{n,k}}{\sqrt{n}} = \mb E \left[ \max_{\bm \sigma \in \Sigma_{n,k}} \la \frac{H_{n,k}(\bm \sigma)}{n} \ra \right] = \frac{\mc H_{n,k}}{n}.
\end{align}
\noindent \underline{\it Step 1: Superadditivity.} We show that for any $n_1, n_2, k_1, k_2$,
\begin{align} \label{super-additivity}
    \mc H_{n_1+n_2, k_1+ k_2} \geq \mc H_{n_1, k_1} + \mc H_{n_2, k_2}. 
\end{align}
To see this, consider the product configuration $\Sigma_{n_1, k_1} \times \Sigma_{n_2, k_2} \subset \Sigma_{n_1+n_2, k_1+k_2}$ and write its elements as $\lb \bm \sigma^{(1)}, \bm \sigma^{(2)} \rb$, with
$\bm \sigma^{(1)} \in \Sigma_{n_1,k_1}$ and $\bm \sigma^{(2)} \in \Sigma_{n_2,k_2}$. We construct two Gaussian processes with appropriate covariance structures so that Slepian's lemma can be applied.

Consider two Gaussian processes on $\Sigma_{n_1, k_1} \times \Sigma_{n_2, k_2}$: 
\begin{itemize}
    \item $X\lb \bm \sigma^{(1)}, \bm \sigma^{(2)} \rb$ is the restriction of $\la H_{n_1+n_2,k_1+k_2} \lb \bm \sigma \rb; \bm \sigma \in \Sigma_{n_1+n_2,k_1+k_2} \ra$ to $\Sigma_{n_1, k_1} \times \Sigma_{n_2, k_2}$.
    \item $Y \lb  \bm \sigma^{(1)}, \bm \sigma^{(2)} \rb =  H_{n_1,k_1} \lb \bm \sigma^{(1)} \rb +  H_{n_2,k_2} \lb \bm \sigma^{(2)} \rb $.
\end{itemize}
Note that for two configurations $\bm \sigma=\lb  \bm \sigma^{(1)}, \bm \sigma^{(2)} \rb \in \Sigma_{n_1, k_1} \times \Sigma_{n_2, k_2}$ and $ \bm \tau= \lb  \bm \tau^{(1)}, \bm \tau^{(2)} \rb \in \Sigma_{n_1, k_1} \times \Sigma_{n_2, k_2}$, we have 
\begin{align*}
    \mb E \left[  X\lb \bm \sigma \rb X\lb \bm \tau \rb \right] 
    = & \lb n_1+n_2 \rb^{-1} \cdot  \mb E \left[ \, \lb \sum_{1\leq i,j \leq n_1+n_2} \sigma_i \sigma_j \cdot A_{ij} \rb \cdot  \lb \sum_{1\leq i,j \leq n_1+n_2} \tau_i \tau_j \cdot A_{ij} \rb \, \right] \\
    =& 2 \lb n_1+n_2 \rb^{-1} \cdot \sum_{i=1}^{n_1+n_2} \sigma_i^2 \tau_i^2 \, + \, \sum_{1 \leq i \neq j \leq n_1+n_2}  2\lb n_1+n_2 \rb^{-1} \cdot \sigma_i\sigma_j \tau_i \tau_j \\
    =& 2\lb n_1+n_2 \rb^{-1}  \cdot  \langle \bm \sigma, \bm \tau \rangle^2 
    = 2\lb n_1+n_2 \rb^{-1}  \cdot \left[ \langle \bm \sigma^{(1)}, \bm \tau^{(1)} \rangle 
    + \langle \bm \sigma^{(2)}, \bm \tau^{(2)} \rangle
    \right]^2,
\end{align*}
 and 
 \begin{align*}
    \mb E \left[  Y\lb \bm \sigma \rb Y\lb \bm \tau \rb \right] 
    = &  \mb E \la \, \left[ H_{n_1,k_1} \lb \bm \sigma^{(1)} \rb +  H_{n_2,k_2} \lb \bm \sigma^{(2)} \rb \right] \cdot \left[ H_{n_1,k_1} \lb \bm \tau^{(1)} \rb +  H_{n_2,k_2} \lb \bm \tau^{(2)} \rb \right] \, \ra \\
    = & \mb E \left[ H_{n_1,k_1} \lb \bm \sigma^{(1)} \rb H_{n_1,k_1} \lb \bm \tau^{(1)} \rb \right]
    + \mb E \left[ H_{n_2,k_2} \lb \bm \sigma^{(2)} \rb H_{n_1,k_1} \lb \bm \tau^{(2)} \rb \right] \\
    =& 2n_1^{-1} \cdot \langle \bm \sigma^{(1)}, \bm \tau^{(1)} \rangle^2 + 2n_2^{-1} \cdot \langle \bm \sigma^{(2)}, \bm \tau^{(2)} \rangle^2.
\end{align*}
Consequently, for all $\bm \sigma, \bm \tau \in \Sigma_{n_1, k_1} \times \Sigma_{n_2, k_2}$,
\[
\mb E \left[ X \lb \bm \sigma \rb^2 \right] =  \mb E \left[ Y \lb \bm \sigma \rb^2 \right] 
\qquad 
\text{and} \qquad 
\mb E \left[  Y\lb \bm \sigma \rb Y\lb \bm \tau \rb \right]
\geq 
\mb E \left[  X\lb \bm \sigma \rb X\lb \bm \tau \rb \right]
\]
where the inequality follows from the convexity of the function $f(x)=x^2$:
\begin{align*}
& 2\lb n_1+n_2 \rb^{-1}  \cdot \left[ \langle \bm \sigma^{(1)}, \bm \tau^{(1)} \rangle 
    + \langle \bm \sigma^{(2)}, \bm \tau^{(2)} \rangle
    \right]^2 \\
= & \, 2(n_1+n_2) \cdot \left[ \frac{n_1}{n_1+n_2} \cdot \frac{\langle \bm \sigma^{(1)}, \bm \tau^{(1)} \rangle}{n_1} + \frac{n_2}{n_1+n_2} \cdot \frac{\langle \bm \sigma^{(2)}, \bm \tau^{(2)} \rangle}{n_2} \right]^2 \\
\leq & \, 2n_1^{-1} \cdot \langle \bm \sigma^{(1)}, \bm \tau^{(1)} \rangle^2 + 2n_2^{-1} \cdot \langle \bm \sigma^{(2)}, \bm \tau^{(2)} \rangle^2.
\end{align*}
By Slepian's lemma, we deduce that
\[
\mc H_{n_1+n_2, k_1+ k_2}  \geq  \mb E \left[ \max_{\bm \sigma \in \Sigma_{n_1, k_1} \times \Sigma_{n_2, k_2}} X \lb \bm \sigma \rb \right] \geq \mb E \left[ \max_{\bm \sigma \in \Sigma_{n_1, k_1} \times \Sigma_{n_2, k_2}} Y \lb \bm \sigma \rb \right] = \mc H_{n_1, k_1} + \mc H_{n_2, k_2}.
\]

\noindent \underline{\it Step 2: Convergence in probability and concavity of $\mc E(c)$.}
For a fixed $c \in (0,1]$, put
\[
g_n(c):= \mc H_{n, \lfloor cn \rfloor}.
\]
By \eqref{super-additivity}, we have $g_{n_1+n_2}(c) \geq g_{n_1}(c) + g_{n_2}(c) $, so Lemma \ref{lem:Fekete} gives the existence of a limit $\mc E(c)$ such that
\[
\lim_{n \to \infty} \frac{\mc H_{n, \lfloor cn \rfloor}}{n} =  \lim_{n \to \infty} \frac{g_n(c)}{n} = \mc E(c). 
\]
Combining the display above with \eqref{equality}, we obtain
\[
\lim_{n \to \infty} \frac{\mb E M_{n,k}}{\sqrt{n}} = \mc E(c).
\]
On the other hand, $M_{n,\lfloor cn \rfloor}$ is the supremum of a Gaussian process with variance profile bounded above by two, so the Gaussian concentration inequality for a Lipschitz functional (see, for example, Theorem 5.8 in \cite{boucheron2013concentration}) yields
\[
\mb P \lb \left| \frac{M_{n,k}}{\sqrt{n}} - \frac{\mb E M_{n,k}}{\sqrt{n}} \right| \geq \ve \rb
\leq \exp \lb -n \ve^2/4 \rb.
\]
Thus, $M_{n,k}/\sqrt{n} $ converges to $\mc E(c)$ almost surely.

To show the concavity of $\mc E(.)$, take $c_1, c_2, \theta \in (0,1]$. Choose two sequences of positive integers $\la a_n, b_n \ra$ such that 
\[
a_n/(a_n+b_n) \to \theta \qquad \text{and} \qquad b_n/(a_n+b_n) \to 1-\theta.
\]
By \eqref{super-additivity}, we have 
\[
\frac{1}{a_n+b_n} \cdot \mc H_{a_n+b_n, \lfloor c_1 a_n \rfloor+\lfloor c_2 b_n \rfloor} 
\geq 
\frac{1}{a_n+b_n} \cdot \mc H_{a_n, \lfloor c_1 a_n \rfloor} + \frac{1}{a_n+b_n} \cdot \mc H_{b_n, \lfloor c_2 b_n \rfloor}
\]
Letting $n \to \infty$, we obtain $\mc E \lb \theta c_1 + (1-\theta) c_2 \rb \geq \theta \cdot \mc E(c_1) + (1-\theta) \mc E(c_2) $. Finally, to prove $\mc E(1)=2$, note that
\[
\frac{\lmax \lb \bm A_{[k]} \rb}{\sqrt{n}} \leq \frac{M_{n,k}}{\sqrt{n}} \leq \frac{\lmax \lb \bm A \rb}{\sqrt{n}}.
\]
Since both sides converge to two in probability, we obtain $\mc E(1)=2$.

\noindent \underline{\it Step 3: Scaling near zero of $\mc E(c)$.}
Define
\[
    b(c):=2\sqrt{c\log(1/c)},
    \qquad 0<c<1.
\]
Take an arbitrary sequence $\la c_j, j \geq 1 \ra$ such that \(c_j\downarrow0\). We will prove that
\[
    \frac{\mathcal E(c_j)}{b(c_j)} \to 1.
\]
For each fixed \(j\), the first
assertion of the theorem gives
\[
    \frac{M_{n,\lfloor c_jn\rfloor}}{\sqrt n}
    \stackrel{\mb P}{\to}
    \mathcal E(c_j)
\]
as $n\to \infty.$ Hence, we may choose a strictly increasing sequence \(n_j\to\infty\)
such that
\[
    c_jn_j\ge j
\]
and, with \(k_j:=\lfloor c_jn_j\rfloor\),
\begin{equation}\label{eq:diagonal-fixed-c}
    \mb P\left(
        \left|
        \frac{M_{n_j,k_j}}{\sqrt{n_j}}
        -
        \mathcal E(c_j)
        \right|
        >
        \frac{b(c_j)}{j}
    \right)
    \le \frac1j.
\end{equation}
Set
\[
    \widetilde c_j:=\frac{k_j}{n_j}.
\]
Then
\[
    k_j\to\infty,
    \qquad
    \widetilde c_j\to  0,
    \qquad
    \frac{\widetilde c_j}{c_j}\to  1.
\]
Consequently,
\begin{equation}\label{eq:b-equivalent}
    \frac{b \lb  \widetilde c_j \rb}{b \lb c_j  \rb}
    =
    \left[
        \frac{\widetilde c_j}{c_j} \cdot
        \frac{\log \lb 1/\widetilde c_j \rb}
             {\log  \lb 1/c_j \rb}
    \right]^{1/2}
    \to1.
\end{equation}
Applying Theorem~\ref{thm:c=0} to the sequence
\((n_j,k_j)\), we obtain
\[
    \frac{M_{n_j,k_j}}
    {2\sqrt{k_j\log(n_j/k_j)}}
    \stackrel{\mb P}{\to}1.
\]
Note that
\[
    \frac{2\sqrt{k_j\log(n_j/k_j)}}{\sqrt{n_j}}
    =
    b\lb \widetilde c_j \rb,
\]
so \eqref{eq:b-equivalent} yields
\[
    \frac{M_{n_j,k_j}/\sqrt{n_j}}{b(c_j)}
    \stackrel{\mb P}{\to}1.
\]
On the other hand, \eqref{eq:diagonal-fixed-c} implies
\[
    \frac{
        M_{n_j,k_j}/\sqrt{n_j}
        -
        \mathcal E(c_j)
    }{
        b(c_j)
    }
    \stackrel{\mb P}{\to}0.
\]
Therefore, $\mc E \lb  c_j\rb/b \lb c_j \rb \to 1$.

{
\par
\medskip
\noindent \underline{\it Step 4: Scaling of $\mc E(c)$ near one.}
As shown at the beginning, $\mc E(c)$ is independent of the fixed diagonal variance $a$, and it is enough to consider $a=2$. The Haar-eigenvector arguments below are applied only after this reduction.

Fix $0<c<1$, put $k=\lfloor cn\rfloor$, and let $Z\sim N(0,1)$. Define $t_c\geq0$ and $\rho_c$ by
\[
\mb P\lb |Z|\geq t_c\rb=c,
\qquad
\rho_c:=\mb E\left[Z^2\mathbf{1}_{\{|Z|\geq t_c\}}\right].
\]
Let $\bm h=(h_1,\ldots,h_n)^\top$, where $h_1,\ldots,h_n$ are i.i.d. $N(0,1)$. For any unit $k$-sparse vectors $\bm x$ and $\bm y$, the GOE covariance identity gives
\[
\mb E\left[\bm x^\top\bm A\bm x-\bm y^\top\bm A\bm y\right]^2
=4\left[1-(\bm x^\top\bm y)^2\right]
\leq
8\left(1-\bm x^\top\bm y\right)
=\mb E\left[2\bm h^\top\bm x-2\bm h^\top\bm y\right]^2.
\]
The Sudakov--Fernique comparison inequality therefore yields
\[
\mb E M_{n,k}
\leq
2\mb E\max_{\substack{\|\bm x\|_2=1\\ \|\bm x\|_0\leq k}}
\bm h^\top\bm x
=2\mb E\left(\sum_{i=1}^k h_{(i)}^2\right)^{1/2},
\]
where $h_{(1)}^2\geq\cdots\geq h_{(n)}^2$ are the decreasing order statistics of $h_1^2,\ldots,h_n^2$. Set $X_i=h_i^2$, and let $X_{1:n}\leq\cdots\leq X_{n:n}$ be the corresponding increasing order statistics. Then, with $J_n(u):=\mathbf{1}_{((n-k)/n,1]}(u)$,
\[
\frac1n\sum_{i=1}^k h_{(i)}^2
=\frac1n\sum_{j=n-k+1}^nX_{j:n}
=\frac1n\sum_{j=1}^nJ_n(j/n)X_{j:n}.
\]
Thus, the trimmed sum is an $L$-statistic, namely, a linear combination of order statistics with deterministic coefficients. Let $F$ be the distribution function of $Z^2$. Since $k/n\to c$, we have $J_n\to\mathbf{1}_{(1-c,1]}$ almost everywhere and $\sup_n\|J_n\|_\infty=1$; moreover, $F^{-1}\in L^1(0,1)$. Hence, the strong law for $L$-statistics \cite[Corollary~2.1]{vanZwet1980} gives
\[
\frac1n\sum_{i=1}^k h_{(i)}^2
\to\int_{1-c}^1F^{-1}(u)\,du
=\mb E\left[Z^2\mathbf{1}_{\{|Z|\geq t_c\}}\right]
=\rho_c
\qquad\text{almost surely}.
\]
The square roots are uniformly integrable because their second moments satisfy
\[
\mb E\left[\frac1n\sum_{i=1}^k h_{(i)}^2\right]
\leq
\mb E\left[\frac1n\sum_{i=1}^n h_i^2\right]=1.
\]
Using the expectation limit obtained in Step 2, we conclude that
\[
\mc E(c)\leq2\sqrt{\rho_c}. 
\]

For the reverse bound, let $\lambda_1\geq\cdots\geq\lambda_n$ be the eigenvalues of $\bm A$, and let $\bm v_1$ be a unit eigenvector associated with $\lambda_1=\lmax(\bm A)$. By orthogonal invariance of the GOE, the eigenvector matrix, modulo irrelevant column signs, is Haar distributed and is independent of the ordered eigenvalues \cite[Corollary~2.5.4]{anderson2010introduction}. Let $S$ contain the indices of the $k$ largest coordinates of $\bm v_1$ in absolute value, and put
\[
\xi_n:=\sum_{i\in S}v_1(i)^2,
\qquad
x_i:=\frac{v_1(i)\mathbf{1}_{\{i\in S\}}}{\sqrt{\xi_n}},
\quad 1\leq i\leq n.
\]
Then $\bm x=(x_i)_{i=1}^n$ is a unit $k$-sparse vector and $M_{n,k}\geq\bm x^\top\bm A\bm x$. The Gaussian representation of a Haar vector and the preceding trimmed-sum law give
\[
\xi_n
\stackrel{d}{=}
\frac{\sum_{i=1}^k h_{(i)}^2}{\sum_{i=1}^n h_i^2}
\stackrel{\mb P}{\to}\rho_c.
\]
Since $0\leq\xi_n\leq1$, we also have $\mb E\xi_n\to\rho_c$. To compute the required conditional expectation, complete $\bm v_1$ to an orthonormal eigenbasis $\bm v_1,\ldots,\bm v_n$ and write
\[
\bm A=\sum_{j=1}^n\lambda_j\bm v_j\bm v_j^\top.
\]
Conditionally on $\bm v_1$ and the eigenvalues, $\bm x$ is fixed,
$\langle\bm x,\bm v_1\rangle^2=\xi_n$, and the remaining eigenvectors form a
Haar orthonormal basis of $\bm v_1^\perp$. Thus, for every $j\geq2$, conditional
rotational invariance gives
\[
\mb E\left[\bm v_j\bm v_j^\top\mid
\bm v_1,\lambda_1,\ldots,\lambda_n\right]
=\frac{I-\bm v_1\bm v_1^\top}{n-1},
\]
and therefore
\[
\mb E\left[(\bm x^\top\bm v_j)^2\mid
\bm v_1,\lambda_1,\ldots,\lambda_n\right]
=\frac{1-\xi_n}{n-1}.
\]
Using the spectral decomposition and summing over $j$ now yields
\[
\mb E\left[\bm x^\top\bm A\bm x\mid \bm v_1,\lambda_1,\ldots,\lambda_n\right]
=\xi_n\lambda_1
+(1-\xi_n)\frac{\operatorname{Tr}(\bm A)-\lambda_1}{n-1}.
\]
The GOE edge theorem, including its $L^1$ version, gives
$\mb E\lambda_1/\sqrt n\to2$ and $\mb E|\lambda_1|=O(\sqrt n)$
\cite[Theorem~2.1.22 and Exercise~2.1.27]{anderson2010introduction}.
Moreover, $\operatorname{Tr}(\bm A)\sim N(0,2n)$, so
$\mb E|\operatorname{Tr}(\bm A)|=2\sqrt{n/\pi}=O(\sqrt n)$. Therefore,
\[
\frac1{\sqrt n}
\left|
\mb E\left[(1-\xi_n)\frac{\operatorname{Tr}(\bm A)-\lambda_1}{n-1}\right]
\right|
\leq
\frac{\mb E|\operatorname{Tr}(\bm A)|+\mb E|\lambda_1|}{\sqrt n(n-1)}
\to0.
\]
Since $\xi_n$ is independent of the eigenvalues, the expectation limit from Step 2 now gives
\[
\mc E(c)\geq2\rho_c.
\]
We have proved
\[
2\rho_c\leq\mc E(c)\leq2\sqrt{\rho_c}.
\]

It remains to expand $\rho_c$ as $c\uparrow1$. We have
\[
1-c
=\mb P\lb |Z|\leq t_c\rb
=\sqrt{\frac2\pi}\,t_c+O(t_c^3),
\qquad
t_c=\sqrt{\frac\pi2}(1-c)+O\lb (1-c)^3\rb,
\]
and hence
\begin{align*}
1-\rho_c
&=\sqrt{\frac2\pi}\int_0^{t_c}x^2e^{-x^2/2}\,dx \\
&=\frac13\sqrt{\frac2\pi}\,t_c^3+O(t_c^5)
=\frac\pi6(1-c)^3+O\lb (1-c)^5\rb.
\end{align*}
The two-sided bound on $\mc E(c)$ therefore implies
\[
2-\mc E(c)
\geq2-2\sqrt{\rho_c}
=\frac\pi6(1-c)^3+O\lb (1-c)^5\rb
\]
and
\[
2-\mc E(c)
\leq2-2\rho_c
=\frac\pi3(1-c)^3+O\lb (1-c)^5\rb.
\]
Thus
\[
\frac\pi6\leq
\liminf_{c\uparrow1}\frac{2-\mc E(c)}{(1-c)^3}
\leq
\limsup_{c\uparrow1}\frac{2-\mc E(c)}{(1-c)^3}
\leq\frac\pi3.
\]
This completes the proof. $\hfill$ $\square$
}


\section{Proof of Theorem \ref{thm:Wishart-limit}}
For a measure $\nu \in \mc P_2 \lb \mb R^k \rb$, write
\begin{align*}
    \mfQ \lb \nu \rb:= \int_{\mb R^k} \bm x \bm x^\top \, d\nu(\bm x).
\end{align*}
For every nonempty \(U\subseteq[k]\), put
\begin{align} \label{D_U(nu)}
    D_U(\nu) := D\left( \nu_U\,\middle\|\,\mu^{\otimes |U|}\right).
\end{align}
Thus
\[
    \mathcal M_{k,\beta,\mu}
    =
    \left\{
        \nu\in\mathcal P_2(\mathbb R^k):
        D_U(\nu)\leq\frac{|U|}{\beta}
        \text{ for every }\varnothing\neq U\subseteq[k]
    \right\}
\]
We first prove the total boundedness and convexity of $\Gamma_{k,\beta,\mu}$. Convexity is immediate because the KL divergence is convex in each argument. To prove total boundedness, take $\nu \in \mc M_{k,\beta,\mu}$. It suffices to derive a uniform bound on $\| \mfQ(\nu) \|_{\rm F}$.

Observe that for two probability measures $\mb P$ and $\mb Q$ and a positive measurable function $h: \mb R^k \to \mb R$, Jensen's inequality yields
\begin{align} \label{Jesen}
\int h \, d \mb P \leq D \lb \mb P \| \mb Q \rb + \log \lb  \int e^h \, d \mb Q \rb.
\end{align}
Applying the preceding inequality with $h(\bm x):= (t_0/2) \cdot \| \bm x \|^2$, $\mb P=\nu$, and $\mb Q=\mu^{\otimes k}$, we obtain
\begin{align*}
    \| \mfQ(\nu) \|_{\rm F} \leq \mbox{tr} \mfQ(\nu) &= \frac{2}{t_0} \cdot \int \frac{t_0}{2} \cdot \| \bm x \|^2 \, d\nu(\bm x) \\
    &\leq \frac{2}{t_0} \cdot \la  D \lb \nu \| \mu^{\otimes k} \rb + \log \left[ \int \exp \lb \frac{t_0\| \bm x \|^2}{2} \rb \, d \mu^{\otimes k}(\bm x) \right] \ra \\
    &\leq \frac{2}{t_0} \cdot \la \frac{k}{\beta} + k\cdot \log \lb \mb E e^{t_0 \xi^2/2} \rb \ra.
\end{align*}
Consequently, total boundedness follows. 

We now prove the convergence in Hausdorff distance. It suffices to show that
\begin{align} \label{upper-bound-Wishart}
    \sup_{\bm G \in \mc C^{\rm Wishart}_{k,\beta,\mu}} d \lb \bm G, \Gamma_{k,\beta,\mu} \rb \stackrel{\mb P}{\to} 0
\end{align}
and 
\begin{align} \label{lower-bound-Wishart}
    \sup_{\bm B \in \Gamma_{k,\beta,\mu}} d \lb  \mc C^{\rm Wishart}_{k,\beta,\mu}, \bm B \rb \stackrel{\mb P}{\to} 0
\end{align}

\subsection{Proof of \eqref{upper-bound-Wishart}} \label{sec:upper-Wishart}
We divide the proof of \eqref{upper-bound-Wishart} into several steps. To establish \eqref{upper-bound-Wishart}, it suffices to prove the weaker statement
\begin{align} \label{upper-Wishart-2}
      \sup_{\bm G \in \mc C^{\rm Wishart}_{k,\beta,\mu}} d \lb \bm G, \Gamma^{(\delta)}_{k,\beta,\mu} \rb \stackrel{\mb P}{\to} 0 \qquad \text{for all fixed $\delta>0$,}
\end{align}
where 
\begin{align}
     \mc M_{k,\beta,\mu}^{(\delta)}  &:= \left\{ \nu\in\mc P_2(\mb R^k): D_U\left( \nu \right) \leq \frac{|U|+\delta}{\beta} \text{ for every } \varnothing\neq U\subseteq[k]\right\},
    \\
    \Gamma_{k,\beta,\mu}^{(\delta)}
    &:=
    \left\{
        \mathsf Q(\nu):
        \nu\in\mc M_{k,\beta,\mu}^{(\delta)}
    \right\}.
\end{align}
To see how \eqref{upper-Wishart-2} implies \eqref{upper-bound-Wishart}, note that if $\nu \in \mc M^{(\delta)}_{k,\beta,\mu}$ and
\[
\widetilde{\nu}:= \frac{1}{1+\delta} \cdot \nu + \frac{\delta}{1+\delta} \cdot \mu^{\otimes k},
\]
then convexity of the KL divergence gives
\begin{align*}
    D\left( \widetilde\nu_U \,\middle\|\,\mu^{\otimes |U|} \right)
    &\leq  \frac{1}{1+\delta} \cdot D\left( \nu_U \,\middle\|\,\mu^{\otimes |U|}\right)
    \leq \frac{|U|+\delta} {\beta(1+\delta)} \leq \frac{|U|}{\beta}
\end{align*}
for all $|U| \geq 1$.

Therefore, we can approximate any $\nu \in \mc M^{(\delta)}_{k,\beta,\mu}$ by $ \widetilde\nu \in  \mc M_{k,\beta,\mu}$ in such a way that
\[
 \left\|  \mathsf Q(\nu) - \mathsf Q(\widetilde\nu) \right\|_{\rm F}
    = \frac{\delta}{1+\delta} \cdot \left\| \mathsf Q(\nu)-\bm I_k \right\|_{\rm F}
    \lesssim \delta
\]
where the last estimate follows from total boundedness. 

Thus, \eqref{upper-Wishart-2} implies \eqref{upper-bound-Wishart} and we may assume that $\nu \in \mc M^{(\delta)}_{k,\beta,\mu} $. To  prove \eqref{upper-Wishart-2}, we work with the closure of $\Gamma^{(\delta)}_{k,\beta,\mu}$, which is more convenient because it is a compact, convex set. To this end, define
\[
\overline{\Gamma}^{(\delta)}_{k,\beta,\mu}:= \overline{\Gamma^{(\delta)}_{k,\beta,\mu}},
\]
which is a compact, convex set. Then \eqref{upper-Wishart-2} is equivalent to
\begin{align} \label{upper-Wishart-3}
    \sup_{ \bm G \in \mc C^{\rm Wishart}_{k,\beta,\mu}} d \lb \bm G, \overline{\Gamma}^{(\delta)}_{k,\beta,\mu} \rb \stackrel{\mb P}{\to} 0 \qquad \text{for all fixed $\delta>0$}.
\end{align}

\underline{\it Step 1: Truncation and duality.} Choose a sufficiently large positive constant $M>0$ and define
\begin{align} \label{mcE}
    \mc E_n(M) := \left\{ \max_{1\leq j\leq n} \frac1p \sum_{a=1}^p\xi_{aj}^2\leq M \right\}.
\end{align}
Using Chernoff's inequality and a union bound, we obtain
\begin{align} \label{truncation}
    \mb P \lb \mc E_n(M)^c  \rb \leq n \cdot \exp \la -p \cdot \left[ \frac{t_0M}{2}- \log  \mb E \exp \lb \frac{t_0 \xi^2}{2}\rb  \right] \ra \to 0
\end{align}
since $p/\log n \to \beta>0$.

For a \(k\)-subset \(S\subset[n]\), put
\begin{align} \label{W_S}
    \bm W_S :=\frac{\bm X_{[S]}^\top\bm X_{[S]}}p.
\end{align}
From \eqref{mcE} and \eqref{truncation}, it follows that, on \(\mc E_n(M)\),
\begin{equation}\label{eq:Wishart-Gram-uniform-bound}
    \|\bm W_S\|_{\rm F}  \leq \operatorname{tr}(\bm W_S) = \sum_{j\in S} \frac1p\sum_{a=1}^p\xi_{aj}^2 \leq kM
    \qquad 
    \text{uniformly over all \(S\)}.
\end{equation}
We next derive a variational representation for $d \lb \bm G, \overline{\Gamma}^{(\delta)}_{k,\beta,\mu} \rb$. Define the support function
\begin{align} \label{h(A)}
h\lb \bm A \rb:= \sup_{\bm X \in \overline{\Gamma}^{(\delta)}_{k,\beta,\mu}} \langle \bm A, \bm X \rangle
\end{align}
where $\langle  \rangle$ is the natural inner product induced by the Frobenius norm.

For $\bm G \in \mb R^{k\times k}$, write
\begin{align}
    d \lb \bm G, \overline{\Gamma}^{(\delta)}_{k,\beta,\mu} \rb = \min_{\bm X \in \overline{\Gamma}^{(\delta)}_{k,\beta,\mu}} \| \bm G - \bm X \|_{\rm F} 
    &= \min_{\bm X \in \overline{\Gamma}^{(\delta)}_{k,\beta,\mu}} \sup_{\| \bm A \|_{\rm F} \leq 1} \langle 
    \bm A, \bm G - \bm X \rangle \nonumber \\
    &= \sup_{\| \bm A \|_{\rm F} \leq 1}  \min_{\bm X \in \overline{\Gamma}^{(\delta)}_{k,\beta,\mu}} \langle 
    \bm A, \bm G - \bm X \rangle \nonumber \\
    &= \sup_{\| \bm A \|_{\rm F} \leq 1} \la  \langle \bm A, \bm G \rangle - h(\bm A) \ra. \label{variatonal-formula}
\end{align}
Here the third equality follows from Sion's minimax theorem \cite{sion1958general}. 

\underline{\it Step 2: $\ve$-net argument.} Put
\[
\eta:= \frac{\ve}{10^3 \cdot \lb kM+ \sup_{\bm Q \in \overline{\Gamma}^{(\delta)}_{k,\beta,\mu}  } \| \bm Q \|_{\rm F} \rb}
\]
and let $\mc N_\ve$ be an $\eta$-net of the set $\la \bm A = \bm A^\top: \| \bm A \|_{\rm F} \leq 1 \ra$.

We show that for every fixed $\ve>0$,
\begin{align} \label{inclusion}
     \mc E_n(M) \cap  \left\{
        \exists\,S\subset[n],\ |S|=k:
        d\lb
            \bm W_S,
            \overline{\Gamma}^{(\delta)}_{k,\beta,\mu}
        \rb >\varepsilon
    \right\}
    \subseteq
    \bigcup_{\bm A\in\mc N_\varepsilon}
    E_{\bm A}
\end{align}
where 
\begin{align}
     E_{\bm A}
    := \mc E_n(M) \cap 
    \left\{
        \text{there exists }
        S\subset[n],\ |S|=k,
        \text{ such that }
        \langle\bm A,\bm W_S\rangle_{\rm F}
        >
        h(\bm A)+\frac{\varepsilon}{2}
    \right\}.
\end{align}
Indeed, \eqref{inclusion} follows from \eqref{variatonal-formula} and the fact that the mapping
\[
\bm A \to \langle \bm A, \bm G \rangle - \sup_{\bm X \in \overline{\Gamma}^{(\delta)}_{k,\beta,\mu}} \langle \bm A, \bm X \rangle
\]
is Lipschitz in the Frobenius norm, with constant no larger than
\[
kM+ \sup_{\bm Q \in \overline{\Gamma}^{(\delta)}_{k,\beta,\mu}  } \| \bm Q \|_{\rm F}.
\]
Thus
\begin{align*}
\mb P\left( \sup_{\bm G\in\mc C_{n,k}^{\rm Wishart}} d\left( \bm G, \overline{\Gamma}^{(\delta)}_{k,\beta,\mu} \right) > \varepsilon  \right)
 &  \leq \mb P\lb \mc E_n(M)^c \rb + \sum_{\bm A\in\mc N_\varepsilon} \mb P(E_{\bm A}) \\
 & \leq \mb P\lb \mc E_n(M)^c \rb + |\mc N_\ve| \cdot \max_{\bm A \in \mc N_\ve} \la \mb P(E_{\bm A}) \ra.
\end{align*}
Note that $|\mc N_\ve|<\infty$. Therefore, to get \eqref{upper-Wishart-3}, it suffices to show that with $h$ as in \eqref{h(A)} and
\[
\widetilde{E}_{\bm A}:= \la  \text{there exists }
        S\subset[n],\ |S|=k,
        \text{ such that }
        \langle\bm A,\bm W_S\rangle_{\rm F}
        >
        h(\bm A)+\frac{\varepsilon}{2} \ra,
\]
we have 
\[
\mb P  \lb \widetilde{E}_{\bm A} \rb \to 0, \qquad \text{for all $\bm A \in \mc N_\ve$}.
\]

\underline{\it Step 3: Bounding the probabilities of the events $\widetilde{E}_{\bm A}$.}
Fix \(\bm A\in\mc N_\varepsilon\), and put $q_{\bm A} :=\mb P\lb \widetilde{E}_{\bm A} \rb$. If \(q_{\bm A}=0\), there is nothing to prove. Assume that \(q_{\bm A}>0\). For a realization $\bm X$, put
\[
\mc V_{\bm A} \lb \bm X \rb:= \la  S \subset [n]: |S|=k, \langle\bm A,\bm W_S\rangle_{\rm F}
        >
        h(\bm A)+\frac{\varepsilon}{2} \ra
\]
where $\bm W_S$ is as in \eqref{W_S}.

We will show that $q_{\bm A}$ tends to zero. Define the probability measure on $\mb R^k$
\begin{align} \label{nu^A}
\nu^{\bm A} \lb B \rb:= \mb E \left[  \frac{1}{\left| \mc V_{\bm A} \lb \bm X \rb \right|}  \cdot \sum_{S \in \mc V_{\bm A} \lb \bm X \rb} \frac 1p \cdot \sum_{i=1}^p \bm 1_{\la \xi_{i,S} \in B  \ra}  \Big|  \widetilde{E}_{\bm A} \right]
\end{align}
for all Borel sets $B \subset \mb R^k$.

The law of $\nu^{\bm A}$ can equivalently be described by the following sampling scheme.
\begin{enumerate}
    \item Conditional on $\bm X= \lb \xi_{ij}; 1\leq i \leq p, 1 \leq j \leq n \rb$ and $\widetilde{E}_{\bm A}$, draw $\tilde{S}$ uniformly from $\mc V_{\bm A} \lb \bm X \rb$.
    \item Independently draw $I \sim \mbox{Unif}[p]$.
    \item Set 
    \[
    \bm Y_{\bm A}:= \lb \xi_{I,s_1}, \dots, \xi_{I,s_k} \rb \in \mb R^k, \qquad \text{where $\tilde{S}=(s_1,s_2,\dots,s_k)$}.
    \]
    Define $\nu^{\bm A}:= \mbox{Law} \lb \bm Y_{\bm A} \rb$.
\end{enumerate}
From the sampling scheme above, it is easy to check that, conditional on $\lb \bm X, \tilde{S} \rb$,
\[
\mb E \left[  \bm Y_{\bm A} \bm Y_{\bm A}^\top \Big| \bm X, \tilde{S} \right] =  \frac{1}{p} \sum_{i=1}^p 
\bm X_{i, \tilde{S}} \bm X_{i, \tilde{S}}^\top = \bm W_{\tilde{S}}.
\]
Thus
\begin{align} \label{upper-contradiction}
\Big\langle \bm A, \mfQ \lb \nu^{\bm A} \rb \Big\rangle  =  \mb E \lb \Big\langle \bm A,  \bm W_{\tilde{S}} \Big\rangle \Big|
\widetilde{E}_{\bm A} \rb \geq h \lb \bm A \rb + \ve/2.
\end{align}
For a subset $U \subset [k]$, let $\nu^{\bm A}_U$ denote the marginal of $\nu^{\bm A}$ on the coordinates in $U$. For all subsets $U \subset [k]$ and Borel sets $B \subset \mb R^{|U|}$, the crude bound in \eqref{nu^A} gives
\begin{align*}
    \nu^{\bm A}_{U}(B) \leq  \binom{n}{|U|} \cdot  \frac{\mu^{\otimes |U|}\, (B)}{q_{\bm A}}.
\end{align*}
Thus
\[
\frac{d \nu_U^{\bm A}}{d \mu^{\otimes |U|}} \ \text{exists for all $|U| \subset [k]$}.
\]
Fix $U \subset [k]$ with $|U|=s$. Let us bound the KL divergence between $\nu^{\bm A}_{U}$ and $\mu^{\otimes s}$. Put
\[
\bm Z_{\bm A, U}:= \lb \bm X_{1,U}, \dots, \bm X_{p,U}  \rb \in  \lb \mb R^{s} \rb^p. 
\]

Let $\rho_{\bm A, U}$ denote the law of $\bm Z_{\bm A, U}$ and $\rho_{\bm A, U, i}$ the law of $\bm X_{i,U}$, which is a probability measure on $\mb R^{s}$. With this notation,
\[
\nu^{\bm A}_U = \frac{1}{p} \sum_{i=1}^p \rho_{\bm A, U, i}.
\]
For a Borel set $C \subset \lb \mb R^{s} \rb^p$, we have 
\begin{align*}
    \rho_{\bm A, U } \lb C \rb = \frac{1}{q_{\bm A}} \cdot \mb E \left[  \bm 1_{\widetilde{E}_{\bm A}} \cdot \frac{1}{\left| \mc V_{\bm A} \lb \bm X \rb \right|}  \cdot \sum_{S \in \mc V_{\bm A}(\bm X)} \bm 1_{\la \bm Z_{\bm A, U} \in C \ra} \right] 
    \leq \binom{n}{s} \cdot \frac{1}{q_{\bm A}} \cdot \mu^{\otimes ps} \lb C \rb. 
\end{align*}
Consequently,
\[
D \lb \rho_{\bm A, U} \Big\| \mu^{\otimes ps}  \rb \leq s\log n + \log \lb 1/ q_{\bm A} \rb.
\]
Moreover, the chain rule for KL divergence (Theorem 2.16(c) in \cite{polyanskiy2025information}) yields
\[
D \lb \rho_{\bm A, U} \Big\| \mu^{\otimes ps}  \rb = D \lb \rho_{\bm A, U} \Big\| \bigotimes_{i=1}^p \rho_{\bm A, U, i}   \rb
+ \sum_{i=1}^p D \lb \rho_{\bm A, U, i} \Big\| \mu^{\otimes s}  \rb \geq \sum_{i=1}^p D \lb \rho_{\bm A, U, i} \Big\| \mu^{\otimes s}  \rb.
\]
Thus
\begin{align*}
    D \lb \nu^{\bm A}_U \Big\| \mu^{\otimes s} \rb \leq \frac{1}{p} \cdot \sum_{i=1}^p  D \lb \rho_{\bm A, U, i} \Big\| \mu^{\otimes s}  \rb
    \leq \frac{1}{p} \cdot D \lb \rho_{\bm A, U} \Big\| \mu^{\otimes ps}  \rb \leq \frac{s \log n}{p} + \frac{1}{p} \cdot \log \lb 1/ q_{\bm A} \rb
\end{align*}
for all $U \subset [k]$ with $|U|=s$, where the first inequality follows from the convexity of KL divergence.

Now suppose that $q_A \geq \widetilde{\ve} $ along a subsequence $\la n_k; k \geq 1 \ra$ for some $\widetilde{\ve}>0$. Then
\[
\frac{s\log \lb n_k \rb}{p_{n_k}} + \frac{1}{p_{n_k}} \cdot \log \lb 1/ q_{\bm A} \rb \leq \frac{s+\delta}{\beta}
\]
for all sufficiently large $n_k$, since $p/\log n \to \beta>0$. This implies that $\nu^{\bm A} \in \mc M^{(\delta)}_{k,\beta,\mu}$ and hence that
\[
\mfQ \lb \nu^{\bm A}  \rb \in  \overline{\Gamma}^{(\delta)}_{k,\beta,\mu} . 
\]
Consequently, by \eqref{variatonal-formula},
\begin{align*}
0 = d \lb \mfQ \lb \nu^{\bm A}  \rb,   \overline{\Gamma}^{(\delta)}_{k,\beta,\mu}  \rb \geq \Big \langle \bm A, \mfQ \lb \nu^{\bm A}  \rb \Big\rangle - h(\bm A),
\end{align*}
which contradicts \eqref{upper-contradiction}. Thus, $q_{\bm A}\to 0$, which yields \eqref{upper-Wishart-3} and, in turn, \eqref{upper-bound-Wishart}.

\subsection{Proof of \eqref{lower-bound-Wishart}}
The proof is similar to that of \eqref{lower}. As in the proof of \eqref{lower}, we may assume strict constraints:
\begin{align} \label{kappa-nu}
    \kappa_\nu := \min_{\varnothing\neq U\subseteq[k]}  \left\{ |U|-\beta D_U(\nu) \right\} >0
\end{align}
where $D_U(\nu)$ is as in \eqref{D_U(nu)}. 

Since $   \overline{\Gamma}^{(\delta)}_{k,\beta,\mu}$ is compact and the distance $d$ is $1$-Lipschitz,  one can use Newey's theorem (Theorem 1 in \cite{newey1991uniform}) to reduce uniform convergence in probability to pointwise convergence in probability. Therefore, it suffices to prove that for every fixed $\eta>0$, 
\begin{align}
    \mb P\left( d\left( \mc C_{n,k}^{\rm Wishart}, \mfQ(\nu)\right)>\eta \right) \to 0
\end{align}
where $\nu \in \mc M_{k,\beta,\mu}$ satisfies \eqref{kappa-nu}.

For every nonempty $U\subseteq[k]$, define
\begin{align} \label{f_U}
    f_U
    :=
    \frac{d\nu_U}{d\mu^{\otimes |U|}}.
\end{align}
We denote $f_{\la 1,2.\dots,k \ra}$ by $f_{[k]}$. Such densities exist because $D_U(\nu)<\infty$. We use the convention
$\log 0=-\infty$. Moreover,
\begin{equation}
\label{eq:Wishart-log-density-integrable}
    \log f_U\in L^1(\nu_U),
    \qquad
    \varnothing\neq U\subseteq[k].
\end{equation}
Indeed, since $D_U(\nu)<\infty$, we have
\begin{align*}
    \int \left| \log f_U  \right| \, d\nu_U &= \int f_U \cdot \left| \log f_U  \right| \, d\mu^{\otimes U} \\
    &= \int_{\la f_U \leq 1 \ra} f_U \cdot \left| \log f_U  \right| \, d\mu^{\otimes U} + \int_{\la f_U > 1 \ra} f_U \cdot \left| \log f_U  \right| \, d\mu^{\otimes U} \\
    & \leq \sup_{x \in (0,1)} x\cdot|\log x| + D_U(\nu) + \sup_{x \in (0,1)} x\cdot|\log x|<\infty.
\end{align*}
Recall the partition $\widetilde{V}$ in \eqref{V-tilde} and $\bm W$ in \eqref{W_S}. For $\bm i= \lb i_1, \dots, i_k \rb \in \widetilde{V}$ and $U \subset [\bm i]$, put
\begin{align*}
    \bm Y^{\bm i}_a &:= \lb \xi_{a, i_1}, \dots, \xi_{a,i_k} \rb, \qquad 1\leq a\leq p\\ 
    \bm W_{\bm i}&:= \bm W_{\la i_1,\dots, i_k \ra} = \frac{1}{p} \cdot \sum_{a=1}^p \bm Y^{\bm i}_a \lb \bm Y^{\bm i}_a \rb^\top \in \mc C^{\rm Wishart}_{n,k}, \\
    L_{\bm i, U} &:= \frac{1}{p} \cdot \sum_{a=1}^p \log f_U \lb \bm Y^{\bm i}_{a, U} \rb.
\end{align*}
Here $\bm Y^{\bm i}_{a, U}$ is the vector formed by taking the coordinates in $\bm Y^{\bm i}_{a}$ that belong to $U$. Pick $\tau \in (0, \kappa_\nu \beta^{-1}/10^3)$ and define the events
\[
E_{\bm i}:= \la \left\| \bm W_{\bm i} - \mfQ \lb \nu \rb  \right\|_{\rm F} < \eta   \ra \cap \lb \bigcap_{ \varnothing \neq U \subset [\bm i]} \la \left| L_{\bm i,U} - D_U(\nu) \right| < \tau  \ra \rb.
\]
As in the proof of \eqref{lower}, we impose additional constraints on $\bm E_{\bm i}$ to control the overlap probabilities. With a slight abuse of notation, paralleling the proof of \eqref{lower}, we also define
\[
N_n = \sum_{\bm i \in \widetilde{V}} \bm 1_{E_{\bm i}}. 
\]
It suffices to show that $\mb P \lb N_n \geq 1 \rb \to 1$. As in the proof of \eqref{lower}, we only need to show that
\[
\frac{\mb E N_n^2}{\lb \mb E N_n \rb^2} -1 \to 0. 
\]

\underline{\it Step 1: Lower bound for $\mb E N_n$.} We first show that $\mb E N_n \to \infty$. Fix an arbitrary $\bm i= \lb i_1,\dots, i_k \rb \in \widetilde{V}$ and write
\begin{align*}
  \mb E N_n =   \left| \widetilde{V} \right| \cdot \mb P_{\mu^{\otimes pk}} \lb E_{\bm i} \rb &\geq m_n^k \cdot  \mb P_{\mu^{\otimes pk}} \lb E_{\bm i} \cap \la \prod_{a=1}^p f_{[\bm i]} \lb \bm Y^{\bm i}_{a} \rb > 0  \ra \rb \\
  &=  m_n^k \cdot \mb E_{\nu^{\otimes p}} \left[ \frac{1}{ \prod_{a=1}^p f_{[\bm i]} \lb \bm Y^{\bm i}_{a} \rb} \cdot \bm 1_{E_{\bm i}} \right]
\end{align*}
where the first line follows from the fact that $|\widetilde{V}|=m_n^k$. 

On $E_{\bm i}$, we have 
\[
 \prod_{a=1}^p f_{[\bm i]} \lb \bm Y^{\bm i}_{a} \rb = \exp \left[ p \cdot L_{\bm i, [i]} \right] \leq \exp \la p \left[ D \lb \nu \Big\| \mu^{\otimes k} \rb + \tau \right] \ra.
\]
Thus
\begin{align*}
\mb E N_n &\geq m_n^k \cdot \exp \la - p \left[ D \lb \nu \Big\| \mu^{\otimes k} \rb + \tau \right] \ra \cdot \mb P_{\nu^{\otimes p}} \lb E_{\bm i} \rb \\
&= \exp \la k\cdot \log (n) - p \left[ D \lb \nu \Big\| \mu^{\otimes k} \rb + \tau \right] + O(1) \ra \cdot \mb P_{\nu^{\otimes p}} \lb E_{\bm i} \rb \\
&= \exp \la \log n \cdot \left[ k - \beta \cdot D \lb \nu \Big\| \mu^{\otimes k} \rb - \tau \beta +o(1)  \right]  \ra  \cdot \mb P_{\nu^{\otimes p}} \lb E_{\bm i} \rb \\
&= \exp \left[ \Omega \lb \log n \rb \right]  \cdot \mb P_{\nu^{\otimes p}} \lb E_{\bm i} \rb 
\end{align*}
where the last line follows from the fact that $\tau \in (0, \kappa_\nu \beta^{-1}/10^3)$ and
\[
 k - \beta \cdot D \lb \nu \Big\| \mu^{\otimes k} \rb \geq \kappa_\nu \geq 2\tau\beta.
\]

By the LLN under $\nu^{\otimes p}$, we have $\mb P_{\nu^{\otimes p}} \lb E_{\bm i} \rb  \to 1$. Thus, $\mb E N_n \to \infty$. Moreover, the argument above also shows that
\[
 \mb P_{\mu^{\otimes pk}} \lb E_{\bm i} \rb \geq \frac{1}{2} \cdot\exp \la - p \left[ D \lb \nu \Big\| \mu^{\otimes k} \rb + \tau \right] \ra
\]
for all sufficiently large $n$.


\underline{\it Step 2: Bounding the overlap probabilities.}
For $\bm i,\bm j\in\widetilde V$, define the overlap set, as in the
proof of \eqref{lower}, by
\[
    U(\bm i,\bm j)
    :=
    \left\{
        r\in[k]:
        i_r=j_r
    \right\}.
\]
As in the proof of \eqref{lower}, any overlap between $\bm i$ and $\bm j$ must occur at the same coordinate. Fix a nonempty proper subset $U\subsetneq[k]$ and suppose that $U(\bm i,\bm j)=U$.

We will show that
\begin{align} \label{overlap-Wishart}
    \mb P \lb  E_{\bm i} \cap E_{\bm j} \rb \leq \exp \la  -p \cdot \left[ 2 D \lb \nu \Big\| \mu^{\otimes k} \rb - D_U(\nu) - 3\tau \right] \ra.
\end{align}
To show \eqref{overlap-Wishart}, write
\[
\mb P \lb  E_{\bm i} \cap E_{\bm j} \rb \leq \mb P \lb S_{\bm i} \cap S_{U} \cap S_{\bm j}  \rb
\]
where 
\begin{align*}
    S_{\bm i}:&= \la \left| L_{\bm i,[\bm i]} - D \lb \nu \Big\| \mu^{\otimes k} \rb \right| < \tau \ra, \\
    S_U:&= \la \left| L_{\bm i, U} - D_U \lb \nu \rb \right| < \tau \ra,    \\
    S_{\bm j} :&= \la \left| L_{\bm j,[\bm j]} - D \lb \nu \Big\| \mu^{\otimes k} \rb \right| < \tau \ra.
\end{align*}
Observe that
\begin{align*}
    \mb P \lb S_{\bm i} \cap S_{U} \cap S_{\bm j}  \rb 
    &\leq \mb P \lb L_{\bm i,[\bm i]} + L_{\bm j,[\bm j]} - L_{\bm i, U} > 2D \lb \nu \Big\| \mu^{\otimes k} \rb - D_U(\nu) -3\tau  \rb \\
    &\leq \exp \la  -p \cdot \left[ 2 D \lb \nu \Big\| \mu^{\otimes k} \rb - D_U(\nu) - 3\tau \right] \ra
    \cdot \mb E \la \exp \left[ p \cdot \lb L_{\bm i,[\bm i]} + L_{\bm j,[\bm j]} - L_{\bm i, U} \rb \right] \ra
\end{align*}
To prove \eqref{overlap-Wishart}, it suffices to show that the expectation in the last display is one. Indeed, with $f_U$ as in \eqref{f_U}, we have
\begin{align*}
    \mb E \la \exp \left[ p \cdot \lb L_{\bm i,[\bm i]} + L_{\bm j,[\bm j]} - L_{\bm i, U} \rb \right] \ra
    =& \mb E \left[ \prod_{a=1}^p
    \frac{
        f_{[\bm i]}\left(\bm Y_a^{\bm i}\right)
        f_{[\bm j]}\left(\bm Y_a^{\bm j}\right)
    }{
        f_U\left(\bm Y_{a,U}^{\bm i}\right)} \right] \\
    =& \prod_{a=1}^p \mb E \left[  \frac{
        f_{[\bm i]}\left(\bm Y_a^{\bm i}\right)
        f_{[\bm j]}\left(\bm Y_a^{\bm j}\right)
    }{
        f_U\left(\bm Y_{a,U}^{\bm i}\right)}  \right] \\
    =& \prod_{a=1}^p \mb E \left[  \frac{
        f_{[\bm i]}\left(  \bm Y_{a,U}^{\bm i}, \bm Y_{a,U^c}^{\bm i} \right)
        f_{[\bm j]}\left( \bm Y_{a,U}^{\bm j}, \bm Y_{a,U^c}^{\bm j} \right)
    }{
        f_U\left(\bm Y_{a,U}^{\bm i}\right) }  \right] \\
    =& \prod_{a=1}^p\mb E \left[ f_U\left(\bm Y_{a,U}^{\bm i}\right)   \right] =1,
\end{align*}
where the last equality follows from the fact that $f_U$ is a density and, in particular, that $f_U$ integrates to one with respect to $\mu^{\otimes |U|}$. Thus, \eqref{overlap-Wishart} holds and, in turn, yields
\begin{align*}
    0 \leq \frac{\mb E N_n^2}{\lb \mb E N_n \rb^2} -1 &\leq \frac{1}{\mb E N_n} + \frac{\sum_{\varnothing\neq U\subsetneq[k]}
    m_n^{2k-|U|}
    \exp\left\{
        -p\left[
            2D\left(
                \nu\,\middle\|\,\mu^{\otimes k}
            \right)
            -
            D_U(\nu)
            -
            3\tau
        \right]
    \right\}}{\lb \mb E N_n \rb^2} \\
    &= o(1) +  \frac{\sum_{\varnothing\neq U\subsetneq[k]}
    m_n^{-|U|}
    \exp\left\{
        -p\left[
            2D\left(
                \nu\,\middle\|\,\mu^{\otimes k}
            \right)
            -
            D_U(\nu)
            -
            3\tau
        \right]
    \right\}}{\mb P_{\mu^{\otimes pk}} \lb E_{\bm i} \rb^2}.  
\end{align*}
The last fraction can be bounded as
\[
\frac{  \sum_{\varnothing\neq U\subsetneq[k]}
    m_n^{-|U|}
    \exp\left\{
        -p\left[
            2D\left(
                \nu\,\middle\|\,\mu^{\otimes k}
            \right)
            -
            D_U(\nu)
            -
            3\tau
        \right]
    \right\}}{\mb P_{\mu^{\otimes pk}} \lb E_{\bm i} \rb^2}
    \leq 
     \sum_{\varnothing\neq U\subsetneq[k]}
     m_n^{-|U|} \cdot \exp \Big\{ p \cdot \left[ D_U(\nu) + 5\tau \right] \Big\}
\]
The last term tends to zero because for every $U \neq \varnothing$,
\begin{align*}
    m_n^{-|U|} \cdot \exp \Big\{ p \cdot \left[ D_U(\nu) + 5\tau \right] \Big\}
    &=
   \exp \Big\{ \log n \cdot  \left[
        -|U|
        +
        \beta D_U(\nu)
        +
        5\beta\tau
        +
        o(1)
    \right]  \Big\} \\
    &\leq \exp \lb -\frac{\kappa_\nu}{2} \cdot \log n \rb \to 0. 
\end{align*}
This completes the proof of \eqref{lower-bound-Wishart}. $\hfill$ $\square$

\section{Proof of Theorem~\ref{p/logn to 0}}

Put
\[
    Y:=\xi^2,
    \qquad
    \kappa:=\kappa_\mu.
\]
By assumption, for every \(\delta>0\), there exists \(x_\delta<\infty\)
such that
\begin{equation}
\label{eq:Wishart-beta-zero-tail-bounds}
    e^{-(\kappa+\delta)x}
    \leq
    \mb P(Y>x)
    \leq
    e^{-(\kappa-\delta)x},
    \qquad
    x\geq x_\delta.
\end{equation}
This implies that
\begin{equation}
\label{eq:Wishart-beta-zero-mgf}
    \mb E e^{tY}<\infty
    \qquad
    \text{for every }0<t<\kappa.
\end{equation}
Fix \(\varepsilon\in(0,1)\). We prove the upper and lower bounds
separately, beginning with the upper bound. For \(1\leq j\leq n\), put
\[
    Z_j
    :=
    \sum_{a=1}^p\xi_{aj}^2.
\]
Choose \(t\) such that
$
    \frac{\kappa}{1+\varepsilon}
    <
    t
    <
    \kappa
$
and write
\[
    M(t):=\mb E e^{tY}<\infty.
\]
By Chernoff's inequality and a union bound,
\begin{align*}
    \mb P\left(
        \max_{1\leq j\leq n}Z_j
        >
        \frac{1+\varepsilon}{\kappa}\log n
    \right)
    &\leq
    n
    \exp\left\{
        -\frac{t(1+\varepsilon)}{\kappa}\log n
    \right\}
    M(t)^p
    \\
    &=
    \exp\left\{
        -
        \left[
            \frac{t(1+\varepsilon)}{\kappa}
            -1
            +o(1)
        \right]\log n
    \right\}
    \to0,
\end{align*}
where we used \(p=o(\log n)\).

For every \(S\subset[n]\) with \(|S|=k\), the Gram matrix
\(\bm X_{[S]}^\top\bm X_{[S]}\) is positive semidefinite, and hence
\begin{align*}
    \lmax\left(
        \bm X_{[S]}^\top\bm X_{[S]}
    \right)
    &\leq
    \operatorname{tr}\left(
        \bm X_{[S]}^\top\bm X_{[S]}
    \right)
    =
    \sum_{j\in S}Z_j
    \leq
    k\max_{1\leq j\leq n}Z_j.
\end{align*}
Consequently,
\[
    \mb P\left(
        \frac{T_{n,k}}{k\log n}
        >
        \frac{1+\varepsilon}{\kappa}
    \right)
    \to0.
\]
We now prove the lower bound. Fix $\varepsilon\in(0,1)$ and put
\[
    u_n
    :=
    \frac{1-\varepsilon}{\kappa_\mu}\log n,
    \qquad
    q_n
    :=
    \mb P\left(
        \xi^2>u_n
    \right).
\]
By the assumed tail limit,
\[
    q_n
    =
    n^{-(1-\varepsilon)+o(1)},
    \qquad
    nq_n
    =
    n^{\varepsilon+o(1)}
    \to\infty.
\]
Let
\[
    N_n
    :=
    \sum_{j=1}^n
    \bm 1_{\{\xi_{1j}^2>u_n\}}
    \sim
    \operatorname{Binomial}(n,q_n).
\]
Since $k$ is fixed, Chebyshev's inequality gives
\[
    \mb P(N_n<k)
    \leq
    \frac{\Var(N_n)}{(nq_n-k)^2}
    \leq
    \frac{nq_n}{(nq_n-k)^2}
    \to0.
\]
On $\{N_n\geq k\}$, choose $S\subset[n]$, $|S|=k$, such that
$\xi_{1j}^2>u_n$ for every $j\in S$. We have
\[
    \bm X_{[S]}^\top\bm X_{[S]}
    =
    \sum_{a=1}^p
    \bm Y_{a,S}\bm Y_{a,S}^\top
    \succeq
    \bm Y_{1,S}\bm Y_{1,S}^\top.
\]
Consequently,
\[
    T_{n,k}
    \geq
    \left\|
        \bm Y_{1,S}
    \right\|_2^2
    =
    \sum_{j\in S}\xi_{1j}^2
    >
    ku_n
    =
    \frac{k(1-\varepsilon)}{\kappa_\mu}\log n.
\]
Therefore,
\[
    \mb P\left(
        \frac{T_{n,k}}{k\log n}
        <
        \frac{1-\varepsilon}{\kappa_\mu}
    \right)
    \to0.
\]
This completes the proof. $\hfill$ $\square$

\section{Proof of Corollaries \ref{co:Wishart-Gaussian} and \ref{co:restricted-isometry}}

\subsection{Proof of Corollary \ref{co:Wishart-Gaussian}}
With the same notation as in Theorem \ref{thm:Wishart-limit}, we only need to show that
\begin{align} \label{mu=Gaussian}
\max_{\bm Q \in \Gamma_{k,\beta, \mu}} \la \lmax \lb \bm Q \rb \ra \Big|_{\mu \equiv N(0,1)} = \lambda_{k,\beta}
\end{align}
where $\lambda_{k,\beta}$ is the unique solution in $(1,\infty)$ of the equation 
\[
\lambda - \log \lambda -1 = \frac{2k}{\beta}
\]
Indeed, a direct calculation shows that the equation above has two positive solutions: one is in $(0,1)$ and one is in $(1,\infty)$.

We begin with the upper bound in \eqref{mu=Gaussian}. Take $t \in (0,1/2)$, $\nu \in \mc M_{k,\beta,\mu}$, and a unit vector $\bm v \in \mb R^k$, and write
\begin{align*}
    t\cdot \bm v^\top \mfQ \lb \nu \rb \bm v = t \int \lb \bm v^\top \bm x \rb^2 \, d \nu(\bm x) 
    &\leq D \lb \nu \Big\| N \lb \bm 0, \bm I_k \rb \rb + \log \la \mb E_{\bm x \sim N \lb \bm 0, \bm I_k \rb} \exp \left[ t \cdot \lb \bm v^\top \bm x \rb^2 \right] \ra \\
    &\leq \frac{k}{\beta} - \frac{\log \lb 1 -2t \rb}{2} 
\end{align*}
where the first inequality follows from Jensen's inequality; see \eqref{Jesen} below for more details. Consequently,
\[
\max_{\bm Q \in \Gamma_{k,\beta, \mu}} \la \lmax \lb \bm Q \rb \ra \Big|_{\mu \equiv N(0,1)}
\leq \inf_{t \in (0,1/2)} \frac{k/\beta - \frac{\log(1-2t)}{2}}{t} .
\]
The infimum in the last display is achieved at some $t_0 \in (0,1/2)$ that satisfies
\[
\frac{t_0}{1-2t_0} - \frac{k}{\beta} + \frac{\log(1-2t_0)}{2} = 0.
\]
Put $\lambda_{k,\beta}= 1/(1-2t_0)$. Then $\lambda_{k,\beta} \geq 1$ and $\lambda_{k,\beta} - \log \lambda_{k,\beta}-1=2k/\beta$, and a direct calculation yields
\[
\frac{k/\beta - \frac{\log(1-2t_0)}{2}}{t_0} = \lambda_{k,\beta}.
\]
Thus
\[
\max_{\bm Q \in \Gamma_{k,\beta, \mu}} \la \lmax \lb \bm Q \rb \ra \Big|_{\mu \equiv N(0,1)} \leq \lambda_{k,\beta}.
\]
For the lower bound in \eqref{mu=Gaussian}, define
\[
\bm v_\star:= \frac{1}{\sqrt{k}} \bm 1_{k}, \qquad \bm Q_\star =  \bm I_k + \lb \lambda_{k,\beta}-1 \rb \bm v_\star \bm v_\star^\top,
\]
and take $\mu_\star = N \lb \bm 0, \bm Q_\star \rb$. 

A direct calculation shows that $\lmax \lb \bm Q_\star \rb =  \lambda_{k,\beta}$. It remains to check that $\mu_\star \in \mc M_{k,\beta,\mu}$. To this end, take $U \subset [k]$ with $|U|=s$ and apply formula (6.32) in \cite{murphy2022probabilistic} to obtain
\[
D \lb N \lb \bm 0, (\bm Q_\star)_U \rb \Big\| N \lb \bm 0, \bm I_s \rb \rb
= \frac 12 \left[ \frac{s}{k} \lb \lambda_{k,\beta} -1 \rb - \log \lb 1 + \frac{s}{k} \lb \lambda_{k,\beta} -1 \rb \rb \right].
\]
Note that the function $g(x)= x - \log(1+x)$ is convex with $g(0)=0$, so $g(sx/k) \leq (s/k)\cdot g(x)$. Therefore, setting $x=\lambda_{k,\beta}-1$ gives
\[
\frac 12 \left[ \frac{s}{k} \lb \lambda_{k,\beta} -1 \rb - \log \lb 1 + \frac{s}{k} \lb \lambda_{k,\beta} -1 \rb \rb \right] 
\leq \frac{s}{2k} \left[ \lambda_{k,\beta} - \log \lambda_{k,\beta}  -1\right] = \frac{s}{\beta}.
\]
Consequently, $\mu_\star \in \mc M_{k,\beta,\mu}$ and we obtain \eqref{mu=Gaussian}. $\hfill$ $\square$

\subsection{Proof of Corollary \ref{co:restricted-isometry}}

Set $F\lb\bm Q\rb:=\|\bm Q-\bm I_k\|_{\op}$. Since
\[
    \left|F\lb\bm Q\rb-F\lb\bm Q'\rb\right|
    \leq\|\bm Q-\bm Q'\|_{\op}
    \leq\|\bm Q-\bm Q'\|_{\rm F},
\]
the map $F$ is $1$-Lipschitz with respect to the Frobenius norm. Therefore,
\[
\left|
\max_{\bm Q\in\mc C^{\rm Wishart}_{n,k}}F\lb\bm Q\rb
-
\sup_{\bm Q\in\Gamma_{k,\beta,\mu}}F\lb\bm Q\rb
\right|
\leq d_{\rm H}\lb\mc C^{\rm Wishart}_{n,k},\Gamma_{k,\beta,\mu}\rb.
\]
The first assertion follows from Theorem \ref{thm:Wishart-limit}, Lemma \ref{convexity}, and the representation \eqref{RIP-constant}.

Now suppose that $\mu=N(0,1)$ and write $\phi(t):=t-\log t-1$. Fix $\nu\in\mc M_{k,\beta,\mu}$, let $\bm Q=\int \bm x\bm x^\top\,d\nu$, and take a unit vector $\bm v\in\mb R^k$. Let $\nu_{\bm v}$ denote the pushforward of $\nu$ under the map $\bm x\mapsto\bm v^\top\bm x$. Data processing for relative entropy and the Gaussian maximum-entropy bound at fixed second moment give
\[
    \frac{1}{2}\phi\lb\bm v^\top\bm Q\bm v\rb
    \leq D\lb\nu_{\bm v}\,\big\|\,N\lb0,1\rb\rb
    \leq D\lb\nu\,\big\|\,N\lb\bm 0,\bm I_k\rb\rb
    \leq \frac{k}{\beta}.
\]
Let $\lambda_-\in(0,1)$ be the other solution of
$\phi(\lambda)=2k/\beta$. It follows that every $\bm Q\in\Gamma_{k,\beta,\mu}$ has all eigenvalues in $[\lambda_-,\lambda_{k,\beta}]$. Moreover,
\[
    1-\lambda_-<\lambda_{k,\beta}-1,
\]
because, with $s:=1-\lambda_-\in(0,1)$,
\[
    \phi(1+s)<\phi(1-s)=\frac{2k}{\beta}
    =\phi\lb\lambda_{k,\beta}\rb,
\]
and $\phi$ is strictly increasing on $(1,\infty)$. Hence
\[
    \|\bm Q-\bm I_k\|_{\op}\leq\lambda_{k,\beta}-1.
\]
To show that this upper bound is attained, put
\[
    \bm v_\star:=\frac{1}{\sqrt{k}}\bm 1_k,
    \qquad
    \bm Q_\star:=\bm I_k+\lb\lambda_{k,\beta}-1\rb\bm v_\star\bm v_\star^\top,
    \qquad
    \nu_\star:=N\lb\bm 0,\bm Q_\star\rb.
\]
For every nonempty $U\subset[k]$, writing $r=|U|$ and using the convexity of $\phi$, we have
\begin{align*}
D\lb \lb\nu_\star\rb_U\,\big\|\,N\lb\bm 0,\bm I_r\rb\rb
&=\frac12\phi\left(1+\frac{r}{k}\lb\lambda_{k,\beta}-1\rb\right)\\
&\leq\frac{r}{2k}\phi\lb\lambda_{k,\beta}\rb
=\frac{r}{\beta}.
\end{align*}
Thus, $\nu_\star\in\mc M_{k,\beta,\mu}$ and $\bm Q_\star\in\Gamma_{k,\beta,\mu}$. Since
\[
    \|\bm Q_\star-\bm I_k\|_{\op}=\lambda_{k,\beta}-1,
\]
the claimed formula for $\Delta_{k,\beta,\mu}$ follows.

Finally, $\lambda_{k,\beta}-1<\delta$ is equivalent to
\[
    \beta>\frac{2k}{\delta-\log\lb1+\delta\rb}.
\]
The two probability limits follow from the convergence in probability proved above. $\hfill$ $\square$

\section{Technical lemmas}

\begin{lemma} \label{convexity}
Suppose \(X\) is the space of all symmetric matrices equipped with the
Frobenius norm. Let \(\mc C,\mc D\subset X\) be nonempty and bounded.
Then, for every convex, \(L\)-Lipschitz functional
\(F:X\to\mb R\),
\[
    \left|
        \sup_{\bm C\in\mc C}F(\bm C)
        -
        \sup_{\bm D\in\mc D}F(\bm D)
    \right|
    \leq
    L\cdot
    d_{\rm H}
    \left(
        \operatorname{conv}(\mc C),
        \operatorname{conv}(\mc D)
    \right).
\]
Here, \(L\)-Lipschitz continuity is understood with respect to the
Frobenius norm, that is,
\[
    \left|F(\bm C)-F(\bm D)\right|
    \leq
    L\|\bm C-\bm D\|_{\rm F}.
\]
\end{lemma}

\noindent\textbf{Proof of Lemma \ref{convexity}.}
Convexity of \(F\) implies
\[
    \sup_{\bm A\in\operatorname{conv}(\mc C)}F(\bm A)
    =
    \sup_{\bm C\in\mc C}F(\bm C).
\]
Indeed, one inequality follows from
\(\mc C\subseteq\operatorname{conv}(\mc C)\). Conversely, if
\[
    \bm A=\sum_{i=1}^m\theta_i\bm C_i,
    \qquad
    \bm C_i\in\mc C,
    \qquad
    \theta_i\geq0,
    \qquad
    \sum_{i=1}^m\theta_i=1,
\]
then
\[
    F(\bm A)
    \leq
    \sum_{i=1}^m\theta_iF(\bm C_i)
    \leq
    \sup_{\bm C\in\mc C}F(\bm C).
\]
Similarly,
\[
    \sup_{\bm B\in\operatorname{conv}(\mc D)}F(\bm B)
    =
    \sup_{\bm D\in\mc D}F(\bm D).
\]

Put
\[
    \mc A:=\operatorname{conv}(\mc C),
    \qquad
    \mc B:=\operatorname{conv}(\mc D),
\]
and let
\[
    \delta:=d_{\rm H}(\mc A,\mc B).
\]
Fix \(\varepsilon>0\). For every \(\bm A\in\mc A\), there exists
\(\bm B_{\bm A}\in\mc B\) such that
\[
    \|\bm A-\bm B_{\bm A}\|_{\rm F}
    \leq
    d(\bm A,\mc B)+\varepsilon
    \leq
    \delta+\varepsilon.
\]
Therefore,
\[
\begin{aligned}
    F(\bm A)
    &\leq
    F(\bm B_{\bm A})
    +
    L\|\bm A-\bm B_{\bm A}\|_{\rm F} \\
    &\leq
    \sup_{\bm B\in\mc B}F(\bm B)
    +
    L(\delta+\varepsilon).
\end{aligned}
\]
Taking the supremum over \(\bm A\in\mc A\), and then letting
\(\varepsilon\downarrow0\), gives
\[
    \sup_{\bm A\in\mc A}F(\bm A)
    -
    \sup_{\bm B\in\mc B}F(\bm B)
    \leq
    L\delta.
\]
Interchanging \(\mc A\) and \(\mc B\) gives the reverse inequality.
Combining the two bounds and using the convex-hull identities proves
the result.
\(\hfill\square\)

\begin{lemma} \label{Gaussian bound}
  With $p_n \lb \bm i, \bm B, \ve \rb$ as in \eqref{p_n}, we have 
    \[
    p_n \lb \bm i, \bm B, \ve \rb \geq \exp \la -\mc I_{\la i_1,\dots,i_k  \ra} \lb \bm B \rb \cdot \log n - C\sqrt{\log n} \ra
    \]
    where $\bm i =  \lb i_1,\dots,i_k \rb$, $k \leq m $ and $C$ is a constant depending only on $ \bm B,  k, a, \ve$.
\end{lemma}

\noindent \textbf{Proof of Lemma \ref{Gaussian bound}.} We first record an elementary bound. Let $G \sim N(0,\sigma^2)$ and $b \in \mb R$. For a fixed $\ve>0$ and all $n$ sufficiently large, we have
\[
\mbox{P} \lb \left| \frac{G}{q_n} - b \right| \leq \ve  \rb \geq \mb P \lb |G-q_nb| \leq 1 \rb \geq \frac{2}{\sqrt{2\pi} \sigma} \cdot \exp \la -\frac{\lb q_n|b| +1 \rb^2}{2\sigma^2}  \ra.
\]
Therefore, using the bound above, 
\begin{align*}
    p_n \lb \bm i, \bm B, \ve \rb &= \prod_{r=1}^k \mb P \la \left| \frac{\bm A_{i_r i_r}}{q_n} - \bm B_{rr} \right| \leq \ve  \ra \times \prod_{1\leq r<s \leq k} \mb P \la  \left| \frac{\bm A_{i_r i_s}}{q_n} - \bm B_{rs} \right| \leq \ve \ra \\
    &\geq \lb \frac{2}{\sqrt{2 \pi a} } \rb^{k} \cdot \lb \frac{2}{\pi} \rb^{k(k-1)/4} \cdot \exp \la -\frac{q_n^2}{2}\cdot \mc I_{\la i_1,\dots,i_k \ra} \lb \bm B \rb - R_n  \ra
\end{align*}
where 
\[
R_n:= \frac{q_n}{a} \cdot \sum_{i=1}^k |B_{ii}| + q_n\cdot \sum_{r<s} |B_{rs}| + \frac{k}{2a} + \frac{k(k-1)}{4} \lesssim \sqrt{\log n}.
\]
for all large $n$.

The constant in brackets depends only on $\bm B, k, a, \ve$. This completes the proof. $\hfill$ $\square$

\begin{lemma} \label{lem: mc R bound}
    Recall $\mc R_z(\rho)$ in \eqref{mc R}. The following statements hold:
    \begin{itemize}
         \item The function $\rho \mapsto R_z(\rho)$ is nondecreasing in $\rho$.
        \item If $0 \leq \rho \leq 1/2$, then for some universal constant $C>0$ and $z\geq 1$,
        \[
        0 \leq R_z(\rho) -1 \leq Cz^2\rho \cdot \exp \lb \rho z^2 \rb.
        \]
        \item For all $0 \leq \rho \leq 1$, there exists a universal constant $C>0$ such that for all $z \geq 1$ 
        \[
         R_z(\rho) \leq Cz \cdot \exp \lb \frac{z^2\rho}{1+\rho} \rb.
        \]    
    \end{itemize}
\end{lemma}

\noindent \textbf{Proof of Lemma \ref{lem: mc R bound}.}
The monotonicity of $\mc R_z$ follows from Slepian's lemma. Assume \(z\ge1\), and write $p_z(\rho):=\mathbb P(G_1>z,G_2>z)$. We use the elementary fact that there exists a constant $c<1$ such that
\begin{align} \label{Gaussian tail bound}
   c^{-1} \cdot z^{-1}e^{-z^2/2} \geq  \overline\Phi(z)\ge c z^{-1}e^{-z^2/2},
\qquad z\ge1.
\end{align}

By Plackett's identity (see Equation (24) in \cite{au2024limit}),
\[
\frac{\partial}{\partial\rho}p_z(\rho)
=
\frac{1}{2\pi\sqrt{1-\rho^2}}
\exp\left\{-\frac{z^2}{1+\rho}\right\}.
\]
Since \(p_z(0)=\overline\Phi(z)^2\), it follows that
\[
\mathcal R_z(\rho)-1
=
\frac{1}{2\pi \cdot \overline\Phi(z)^2}
\int_0^\rho
\frac{1}{\sqrt{1-r^2}}
\exp\left\{-\frac{z^2}{1+r}\right\}\,dr.
\]
In particular, \(\mathcal R_z(\rho)\ge1\). 
Moreover, if \(0\le\rho\le1/2\), \eqref{Gaussian tail bound} gives
\[
\begin{aligned}
\mathcal R_z(\rho)-1
\le
Cz^2
\int_0^\rho
\exp\left\{
z^2-\frac{z^2}{1+r}
\right\}\,dr 
&=Cz^2 \cdot 
\int_0^\rho
\exp\left\{
\frac{r}{1+r} \cdot z^2
\right\}\,dr \\
&\le
Cz^2 \cdot \rho \cdot \exp\lb \rho z^2 \rb.
\end{aligned}
\]
For the final bound, observe that
\[
\{G_1>z,G_2>z\}
\subset
\left\{
\frac{G_1+G_2}{\sqrt{2(1+\rho)}}
>
z\sqrt{\frac{2}{1+\rho}}
\right\}.
\]
Thus, \eqref{Gaussian tail bound} gives
\begin{align*}
    \mc R_z(\rho) = \frac{p_z(\rho)}{p_z(0)} \leq \overline\Phi\left(z\sqrt{\frac{2}{1+\rho}} \, \right) \cdot \frac{1}{\overline{\Phi}(z)^2}
    &\leq C \cdot \frac{e^{z^2}}{z^{-2}} \cdot z^{-1} \cdot \sqrt{\frac{1+\rho}{2}} \exp \lb -\frac{z^2}{1+\rho} \rb \\
    &\leq Cz \cdot \exp \lb \frac{\rho z^2}{1+\rho} \rb.
\end{align*}
This completes the proof. $\hfill$ $\square$

\renewcommand{\bibname}{References}
\bibliographystyle{plainnat}
\bibliography{paper-ref}

@article{cai2014sparse,
    author = {Cai, T. Tony and Zhang, Anru},
    title = {Sparse Representation of a Polytope and Recovery of Sparse Signals and Low-Rank Matrices},
    journal = {IEEE Transactions on Information Theory},
    volume = {60},
    number = {1},
    pages = {122--132},
    year = {2014}
}

@article{JiangCai12,
    author = {Cai, T. and Jiang, T.},
    title = {Phase transition in limiting distributions of coherence of high-dimensional random matrices},
    journal = {J. Multivariate Anal. 107 24–39.},
    year =  {2012},
}

@article{supp,
  title={Supplement to ``Time uniform concentration bounds for iterative algorithms"},
  author={Pham, T. and Rinaldo, A. and Sarkar, P.},
journal={},
  year={2025}
}

@article{arratia1989two,
  title={Two moments suffice for {P}oisson approximations: the {C}hen-{S}tein method},
  author={Arratia, Richard and Goldstein, Larry and Gordon, Louis},
  journal={The Annals of Probability},
  pages={9--25},
  year={1989},
  publisher={JSTOR}
}

@inproceedings{feng2026principal,
  title={Principal minors of {G}aussian orthogonal ensemble},
  author={Feng, Renjie and Tian, Gang and Wei, Dongyi and Yao, Dong},
  booktitle={Forum Mathematicum},
  volume={38},
  number={4},
  pages={1187--1209},
  year={2026},
  organization={De Gruyter}
}

@article{cai2021asymptotic,
  title={Asymptotic analysis for extreme eigenvalues of principal minors of random matrices},
  author={Cai, T Tony and Jiang, Tiefeng and Li, Xiaoou},
  journal={The Annals of Applied Probability},
  volume={31},
  number={6},
  pages={2953--2990},
  year={2021},
  publisher={JSTOR}
}

@article{jiang2024largest,
  title={Largest eigenvalues of principal minors of deformed {G}aussian orthogonal ensembles and Wishart matrices},
  author={Jiang, Tiefeng and Qi, Yongcheng},
  journal={arXiv preprint arXiv:2410.15160},
  year={2024}
}

@article{iusem2010distances,
  title={Distances between closed convex cones: old and new results},
  author={Iusem, Alfredo and Seeger, Alberto},
  journal={J. Convex Anal},
  volume={17},
  number={3-4},
  pages={1033--1055},
  year={2010}
}

@article{au2024limit,
  title={A limit formula and a series expansion for the bivariate Normal tail probability},
  author={Au, Siu-Kui},
  journal={Statistics and Computing},
  volume={34},
  number={5},
  pages={152},
  year={2024},
  publisher={Springer}
}

@book{polyanskiy2025information,
  title={Information theory: From {C}oding to {L}earning},
  author={Polyanskiy, Yury and Wu, Yihong},
  year={2025},
  publisher={Cambridge {U}niversity {P}ress}
}

@book{anderson2010introduction,
  title={An {I}ntroduction to {R}andom {M}atrices},
  author={Anderson, Greg W and Guionnet, Alice and Zeitouni, Ofer},
  year={2010},
  publisher={Cambridge University Press}
}

@book{murphy2022probabilistic,
  title={Probabilistic {M}achine {L}earning: an {I}ntroduction},
  author={Murphy, Kevin P},
  year={2022},
  publisher={MIT press}
}

@book{BaiSilverstein2010,
  author    = {Bai, Zhidong and Silverstein, Jack W.},
  title     = {Spectral Analysis of Large Dimensional Random Matrices},
  edition   = {2},
  publisher = {Springer},
  address   = {New York},
  year      = {2010}
}

@article{BrycDemboJiang2006,
  author  = {Bryc, W{\l}odzimierz and Dembo, Amir and Jiang, Tiefeng},
  title   = {Spectral Measure of Large Random {Hankel}, {Markov} and {Toeplitz} Matrices},
  journal = {The Annals of Probability},
  year    = {2006},
  volume  = {34},
  number  = {1},
  pages   = {1--38}
}

@article{DiaconisEvans2001,
  author  = {Diaconis, Persi and Evans, Steven N.},
  title   = {Linear Functionals of Eigenvalues of Random Matrices},
  journal = {Transactions of the American Mathematical Society},
  year    = {2001},
  volume  = {353},
  number  = {7},
  pages   = {2615--2633}
}

@article{Bai1999,
  author  = {Bai, Zhidong},
  title   = {Methodologies in Spectral Analysis of Large Dimensional Random Matrices, a Review},
  journal = {Statistica Sinica},
  year    = {1999},
  volume  = {9},
  number  = {3},
  pages   = {611--677}
}

@article{candes2008restricted,
  title={The restricted isometry property and its implications for compressed sensing},
  author={Candes, Emmanuel J},
  journal={Comptes {R}endus {M}athematique},
  volume={346},
  number={9-10},
  pages={589--592},
  year={2008},
  publisher={Elsevier}
}

@article{moghaddam2005spectral,
  title={Spectral bounds for sparse {PCA}: Exact and greedy algorithms},
  author={Moghaddam, Baback and Weiss, Yair and Avidan, Shai},
  journal={Advances in {N}eural {I}nformation {P}rocessing {S}ystems},
  volume={18},
  year={2005}
}

@article{d2004direct,
  title={A direct formulation for sparse {PCA} using semidefinite programming},
  author={d'Aspremont, Alexandre and Ghaoui, Laurent and Jordan, Michael and Lanckriet, Gert},
  journal={Advances in {N}eural {I}nformation {P}rocessing {S}ystems},
  volume={17},
  year={2004}
}

@article{zhang2008sparsity,
  title={The sparsity and bias of the {LASSO} selection in high-dimensional linear regression},
  author={Zhang, Cun-Hui and Huang, Jian},
  journal={Annals of Statistics},
  volume={36},
  number={4},
  pages={1567--1594},
  year={2008}
}

@article{hu2023extreme,
  title={Extreme eigenvalues of principal minors of random matrices with moment conditions},
  author={Hu, Jianwei and Keita, Seydou and Fu, Kang},
  journal={Journal of the Korean Statistical Society},
  volume={52},
  number={3},
  pages={715--735},
  year={2023},
  publisher={Springer}
}

@book{boucheron2013concentration,
  author    = {Boucheron, St{\'e}phane and Lugosi, G{\'a}bor and
               Massart, Pascal},
  title     = {Concentration Inequalities:
               A Nonasymptotic Theory of Independence},
  publisher = {Oxford University Press},
  year      = {2013}
}

@article{sion1958general,
  title   = {On General Minimax Theorems},
  author  = {Sion, Maurice},
  journal = {Pacific Journal of Mathematics},
  volume  = {8},
  number  = {1},
  pages   = {171--176},
  year    = {1958}
}

@book{dembo2010large,
  author    = {Dembo, Amir and Zeitouni, Ofer},
  title     = {Large Deviations Techniques and Applications},
  edition   = {2},
  publisher = {Springer},
  address   = {Berlin, Heidelberg},
  year      = {2010}
}

@article{chen2013aizenman,
  author  = {Chen, Wei-Kuo},
  title   = {The {Aizenman--Sims--Starr} Scheme and {Parisi} Formula
             for Mixed {$p$}-Spin Spherical Models},
  journal = {Electronic Journal of Probability},
  volume  = {18},
  number  = {94},
  pages   = {1--14},
  year    = {2013}
}

@article{chen2017parisiSpherical,
  author  = {Chen, Wei-Kuo and Sen, Arnab},
  title   = {{Parisi} Formula, Disorder Chaos and Fluctuation for the
             Ground State Energy in the Spherical Mixed {$p$}-Spin Models},
  journal = {Communications in Mathematical Physics},
  volume  = {350},
  number  = {1},
  pages   = {129--173},
  year    = {2017}
}

@article{vanZwet1980,
  author  = {van Zwet, Willem R.},
  title   = {A Strong Law for Linear Functions of Order Statistics},
  journal = {The Annals of Probability},
  volume  = {8},
  number  = {5},
  pages   = {986--990},
  year    = {1980}
}

@article{newey1991uniform,
  author  = {Newey, Whitney K.},
  title   = {Uniform Convergence in Probability and Stochastic Equicontinuity},
  journal = {Econometrica},
  volume  = {59},
  number  = {4},
  pages   = {1161--1167},
  year    = {1991}
}

\end{document}